\documentclass[reqno]{amsart}
\usepackage{amssymb,latexsym,amsmath,amsthm,enumerate,amsbsy}
\usepackage[mathscr]{eucal}
\usepackage{framed,color,graphicx}
\usepackage{mathrsfs}
\usepackage[all]{xy}
\usepackage{tikz}
\usepackage{xcolor}
\usepackage{cite}
\usepackage[misc]{ifsym}
\usetikzlibrary{positioning,decorations.pathreplacing,patterns,decorations.pathmorphing}
\tikzset{%
element/.style={draw, shape=circle, fill=white, inner sep=1.4pt}
}
\DeclareSymbolFont{bbold}{U}{bbold}{m}{n}
\DeclareSymbolFontAlphabet{\mathbbold}{bbold}

\theoremstyle{plain}
\newtheorem{thm}{Theorem}[section]
\newtheorem{lem}[thm]{Lemma}
\newtheorem{cor}[thm]{Corollary}
\newtheorem{pro}[thm]{Proposition}

\theoremstyle{definition}

\newtheorem{remark}[thm]{Remark}

\newcommand{\up}[1]{\textup{#1}}

\newcommand{\bp}{\mathbf{p}}
\newcommand{\bq}{\mathbf{q}}

\newcommand{\bu}{\mathbf{u}}
\newcommand{\bv}{\mathbf{v}}
\newcommand{\bw}{\mathbf{w}}

\begin{document}

\title[The finite basis problem for ai-semirings of order four]
{The finite basis problem for additively idempotent semirings of order four, III}

\author{Miaomiao Ren}
\address{School of Mathematics, Northwest University, Xi'an, 710127, Shaanxi, P.R. China}
\email{miaomiaoren@yeah.net}

\author{Zexi Liu}
\address{School of Mathematics, Northwest University, Xi'an, 710127, Shaanxi, P.R. China}
\email{zexiliu@yeah.net}

\author{Mengya Yue}
\address{School of Mathematics, Northwest University, Xi'an, 710127, Shaanxi, P.R. China}
\email{myayue@yeah.net}

\author{Yizhi Chen}
\address{ Department of Mathematics and Statistics, Huizhou University, Huizhou 516007, Guangdong, P.R. China}
\email{yizhichen1980@126.com}

\thanks{Miaomiao Ren is supported by National Natural Science Foundation of China (12371024).
Mengya Yue, corresponding author, is supported by National Natural Science Foundation of China (12571020).
Yizhi Chen is supported by Guangdong Basic and Applied Basic Research Foundation (2023A1515011690),
and Characteristic Innovation Project of Guangdong Provincial Department of Education (2023KTSCX145).
}

\begin{abstract}
We study the finite basis problem for $4$-element additively idempotent semirings whose
additive reducts have two minimal elements and one coatom. Up to isomorphism, there are $112$ such algebras.
We show that $106$ of them are finitely based and the remaining ones are nonfinitely based.
\end{abstract}

\maketitle

\section{Introduction and preliminaries}
An \emph{additively idempotent semiring} (ai-semiring for short)
is an algebra $(S, +, \cdot)$ equipped with two binary operations $+$ and $\cdot$ satisfying the following axioms:
\begin{itemize}
\item The additive reduct $(S, +)$ forms a commutative idempotent semigroup;

\item The multiplicative reduct $(S, \cdot)$ forms a semigroup;

\item The distributive laws hold:
\[
x(y+z)\approx xy+xz,~\textrm{and}~(x+y)z\approx xz+yz.
\]
\end{itemize}
An ai-semiring variety is \emph{finitely based} if it can be axiomatized by a finite set of identities;
otherwise, it is \emph{nonfinitely based}.
An ai-semiring $S$ is finitely based or nonfinitely based according to whether the variety $\mathsf{V}(S)$ it generates is finitely based or not.

The finite basis problem for a class of ai-semirings concerns
the classification of its members with respect to the property of being finitely based.
The present paper continues the line of the research initiated in \cite{rlzc, yrzs},
focusing on the finite basis problem for $4$-element ai-semirings whose
additive reducts have two minimal elements and one coatom.
For futher background and motivation, we refer the reader to \cite{rlzc, yrzs}.

There are, up to isomorphism, precisely $6$ ai-semirings of order two,
which are denoted by $L_2$, $R_2$, $M_2$, $D_2$, $N_2$ and $T_2$.
The solution of the equational problem for these algebras is provided in \cite[Lemma 1.1]{sr}
and will be repeatedly used in subsequent sections.
For the ai-semirings of order three, there are $61$ isomorphism types,
indexed as $S_i$ for $1 \leq i \leq 61$.
A comprehensive description of these algebras is available in \cite{zrc}.

Up to isomorphism, there are exactly $866$ ai-semirings of order four.
These algebras are categorized into five distinct types based on their additive orders, as illustrated in \cite[Figure 1]{rlzc}.
The finite basis problem for semirings in the first and second types has been resolved by
Ren et al. \cite{rlzc} and Yue et al. \cite{yrzs}, respectively.
The present paper addresses the finite basis problem for semirings in the third type,
which consists of all $4$-element ai-semirings whose additive reducts have two minimal elements and one coatom.
There are, up to isomorphism, $112$ such algebras, which are denoted by $S_{(4, k)}$ for $276\leq k \leq 387$.
We assume that the carrier set of these semirings is $\{1, 2, 3, 4\}$.
Their Cayley tables for addition are determined by Figure \ref{figure01},
and their Cayley tables for multiplication are listed in Table \ref{tb1}.

\setlength{\unitlength}{0.9cm}
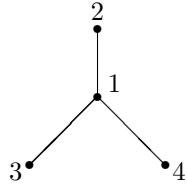
\begin{figure}[htbp]\label{figure01}
\begin{picture}(35, 2.25)
\put(7,1){\line(-1,-1){1}}
\put(7,1){\line(0,1){1}}
\put(7,1){\line(1,-1){1}}
\put(7,1){\line(-1,-1){1}}

\multiput(7,2)(1,-1){1}{\circle*{0.1}}
\multiput(7,1)(-1,-1){1}{\circle*{0.1}}
\multiput(8,0)(1,1){1}{\circle*{0.1}}
\multiput(6,0)(1,1){1}{\circle*{0.1}}

\put(7,2.25){\makebox(0,0){$2$}}
\put(7,0.2){\makebox(0.5,2){$1$}}
\put(5.8,-0.6){\makebox(0,1){$3$}}
\put(8.2,-0.1){\makebox(0,0){$4$}}
\end{picture}
\caption{The additive order of $S_{(4, k)}$, $276\leq k \leq 387$}
\end{figure}

This paper establishes the following main result, whose proof will be developed in the subsequent sections.
\begin{thm}
The only nonfinitely based ai-semirings in $\{S_{(4, k)} \mid 276\leq k \leq 387\}$ are
$S_{(4, 282)}$, $S_{(4, 293)}$, $S_{(4, 304)}$, $S_{(4, 326)}$, $S_{(4, 335)}$, and $S_{(4, 359)}$.
\end{thm}

Let $X$ denote a countably infinite set of variables and $X^+$
the free semigroup over $X$. By distributivity,
every ai-semiring term over $X$ can be written as a finite sum of words in $X^+$.
An \emph{ai-semiring identity} over $X$ is an
expression of the form
\[
\bu\approx \bv,
\]
where $\bu$ and $\bv$ are ai-semiring terms over $X$.
From \cite[Theorem 2.5]{kp}, the ai-semiring $(P_f(X^+), \cup, \cdot)$, consisting of all non-empty finite subsets of $X^+$,
is free in the variety $\mathbf{AI}$ of all ai-semirings over $X$.
So we may write
\[
\{\bu_i \mid 1 \leq i \leq k\}\approx \{\bv_j \mid 1 \leq j \leq \ell\}
\]
to denote the ai-semiring identity
\[
\bu_1+\cdots+\bu_k\approx \bv_1+\cdots+\bv_\ell.
\]

Let $S$ be an ai-semiring and $\bu\approx \bv$ an ai-semiring identity.
We say that \emph{$S$ satisfies $\bu\approx \bv$} or \emph{$\bu\approx \bv$ holds in $S$},
if $\varphi(\bu)=\varphi(\bv)$ for every semiring homomorphism $\varphi: P_f(X^+) \to S$.
Note that $\varphi$ is uniquely determined by its values $\{\varphi(x)\mid x\in X\}$,
since $P_f(X^+)$ is generated by $X$.

We now fix some notation that will be used repeatedly.
Let $\bw$ be a word in $X^+$ and $x$ a variable in $X$. Then
\begin{itemize}
\item $t(\bw)$ denotes the last variable that occurs in $\bw$;

\item $c(\bw)$ denotes the set of variables that occur in $\bw$;

\item $\ell(\bw)$ denotes the number of variables occurring in $\bw$ counting multiplicities;

\item $m(x, \bw)$ denotes the number of occurrences of $x$ in $\bw$;

\item $p(\bw)$ denotes the subword of $\bw$ such that $\bw=p(\bw)t(\bw)$.
\end{itemize}

Let $\bu$ be an ai-semiring term such that $\bu=\bu_1+\bu_2+\cdots+\bu_n$,
where $\bu_i \in X^+$, $1 \leq i \leq n$.
Let $\bq$ be a nonempty word, and let $k$ be a positive integer. Then
\begin{itemize}
\item $t(\bu)$ denotes the set $\{t(\bu_i) \mid 1\leq i \leq n\}$;

\item $c(\bu)$ denotes the union of the sets $c(\bu_i)$ for $1\leq i \leq n$;

\item $L_{\geq k}(\bu)$ denotes the set $\{\bu_i \in \bu \mid \ell(\bu_i)\geq k\}$;

\item $L_{\leq k}(\bu)$ denotes the set $\{\bu_i \in \bu \mid \ell(\bu_i)\leq k\}$;

\item $L_k(\bu)$ denotes the set $\{\bu_i \in \bu \mid \ell(\bu_i)= k\}$;

\item $T_{\bq}(\bu)$ denotes the set $\{\bu_i \in \bu \mid t(\bu_i)=t(\bq)\}$;

\item $D_{\bq}(\bu)$ denotes the set $\{\bu_i \in \bu \mid c(\bu_i)\subseteq c(\bq)\}$;

\item $M_{1}(\bq)$ denotes the set $\{x \in c(\bq) \mid m(x, \bq)=1\}$.
\end{itemize}

\begin{table}[htbp]
\caption{The multiplicative tables of $4$-element ai-semirings whose additive reducts have two minimal elements and one coatom} \label{tb1}
\begin{tabular}{cccccc}
\hline
Semiring & $\cdot$ & Semiring & $\cdot$ & Semiring & $\cdot$\\
\hline
$S_{(4, 276)}$
&
\begin{tabular}{cccc}
1 & 1 & 1 & 1\\
1 & 1 & 1 & 1 \\
1 & 1 & 1 & 1 \\
1 & 1 & 1 & 1 \\
\end{tabular}
&
$S_{(4, 277)}$
&
\begin{tabular}{cccc}
1 & 1 & 1 & 1\\
1 & 1 & 1 & 1 \\
1 & 1 & 1 & 1 \\
1 & 1 & 1 & 3 \\
\end{tabular}
&
$S_{(4, 278)}$
&
\begin{tabular}{cccc}
1 & 1 & 1 & 1\\
1 & 1 & 1 & 1 \\
1 & 1 & 1 & 1 \\
1 & 1 & 1 & 4 \\
\end{tabular}\\
\hline

$S_{(4, 279)}$
&
\begin{tabular}{cccc}
1 & 1 & 1 & 1\\
1 & 1 & 1 & 1 \\
1 & 1 & 1 & 3 \\
1 & 1 & 1 & 4 \\
\end{tabular}
&
$S_{(4, 280)}$
&
\begin{tabular}{cccc}
1 & 1 & 1 & 4\\
1 & 1 & 1 & 4 \\
1 & 1 & 1 & 4 \\
1 & 1 & 1 & 4 \\
\end{tabular}
&
$S_{(4, 281)}$
&
\begin{tabular}{cccc}
1 & 1 & 1 & 1\\
1 & 1 & 1 & 1 \\
1 & 1 & 1 & 1 \\
1 & 1 & 3 & 4 \\
\end{tabular}\\
\hline

$S_{(4, 282)}$
&
\begin{tabular}{cccc}
1 & 1 & 1 & 1\\
1 & 1 & 1 & 1 \\
1 & 1 & 1 & 3 \\
1 & 1 & 3 & 4 \\
\end{tabular}
&
$S_{(4, 283)}$
&
\begin{tabular}{cccc}
1 & 1 & 1 & 1\\
1 & 1 & 1 & 1 \\
1 & 1 & 3 & 1 \\
1 & 1 & 1 & 4 \\
\end{tabular}
&
$S_{(4, 284)}$
&
\begin{tabular}{cccc}
1 & 1 & 1 & 4\\
1 & 1 & 1 & 4 \\
1 & 1 & 3 & 4 \\
1 & 1 & 1 & 4 \\
\end{tabular}\\
\hline

$S_{(4, 285)}$
&
\begin{tabular}{cccc}
1 & 1 & 1 & 1\\
1 & 1 & 1 & 1 \\
1 & 1 & 3 & 4 \\
1 & 1 & 4 & 3 \\
\end{tabular}
&
$S_{(4, 286)}$
&
\begin{tabular}{cccc}
1 & 1 & 3 & 4\\
1 & 1 & 3 & 4 \\
1 & 1 & 3 & 4 \\
1 & 1 & 3 & 4 \\
\end{tabular}
&
$S_{(4, 287)}$
&
\begin{tabular}{cccc}
1 & 1 & 1 & 1\\
1 & 2 & 1 & 1 \\
1 & 1 & 1 & 1 \\
1 & 1 & 1 & 1 \\
\end{tabular}\\
\hline
$S_{(4, 288)}$
&
\begin{tabular}{cccc}
1 & 1 & 1 & 1\\
1 & 2 & 1 & 1\\
1 & 1 & 1 & 1\\
1 & 1 & 1 & 3 \\
\end{tabular}
&
$S_{(4, 289)}$
&
\begin{tabular}{cccc}
1 & 1 & 1 & 1\\
1 & 2 & 1 & 1 \\
1 & 1 & 1 & 1 \\
1 & 1 & 1 & 4 \\
\end{tabular}
&
$S_{(4, 290)}$
&
\begin{tabular}{cccc}
1 & 1 & 1 & 1\\
1 & 2 & 1 & 1 \\
1 & 1 & 1 & 3 \\
1 & 1 & 1 & 4 \\
\end{tabular}\\
\hline

$S_{(4, 291)}$
&
\begin{tabular}{cccc}
1 & 1 & 1 & 4\\
1 & 2 & 1 & 4 \\
1 & 1 & 1 & 4 \\
1 & 1 & 1 & 4 \\
\end{tabular}
&
$S_{(4, 292)}$
&
\begin{tabular}{cccc}
1 & 1 & 1 & 1\\
1 & 2 & 1 & 1 \\
1 & 1 & 1 & 1 \\
1 & 1 & 3 & 4 \\
\end{tabular}
&
$S_{(4, 293)}$
&
\begin{tabular}{cccc}
1 & 1 & 1 & 1\\
1 & 2 & 1 & 1 \\
1 & 1 & 1 & 3 \\
1 & 1 & 3 & 4 \\
\end{tabular}\\
\hline

$S_{(4, 294)}$
&
\begin{tabular}{cccc}
1 & 1 & 1 & 1\\
1 & 2 & 1 & 1 \\
1 & 1 & 3 & 1 \\
1 & 1 & 1 & 4 \\
\end{tabular}
&
$S_{(4, 295)}$
&
\begin{tabular}{cccc}
1 & 1 & 1 & 4\\
1 & 2 & 1 & 4 \\
1 & 1 & 3 & 4 \\
1 & 1 & 1 & 4 \\
\end{tabular}
&
$S_{(4, 296)}$
&
\begin{tabular}{cccc}
1 & 1 & 1 & 1\\
1 & 2 & 1 & 1 \\
1 & 1 & 3 & 4 \\
1 & 1 & 4 & 3 \\
\end{tabular}\\
\hline

$S_{(4, 297)}$
&
\begin{tabular}{cccc}
1 & 1 & 3 & 4\\
1 & 2 & 3 & 4 \\
1 & 1 & 3 & 4 \\
1 & 1 & 3 & 4 \\
\end{tabular}
&
$S_{(4, 298)}$
&
\begin{tabular}{cccc}
1 & 2 & 1 & 1\\
1 & 2 & 1 & 1 \\
1 & 2 & 1 & 1 \\
1 & 2 & 1 & 1 \\
\end{tabular}
&
$S_{(4, 299)}$
&
\begin{tabular}{cccc}
1 & 2 & 1 & 1\\
1 & 2 & 1 & 1 \\
1 & 2 & 1 & 1 \\
1 & 2 & 1 & 3 \\
\end{tabular}\\
\hline

$S_{(4, 300)}$
&
\begin{tabular}{cccc}
1 & 2 & 1 & 1\\
1 & 2 & 1 & 1 \\
1 & 2 & 1 & 1 \\
1 & 2 & 1 & 4 \\
\end{tabular}
&
$S_{(4, 301)}$
&
\begin{tabular}{cccc}
1 & 2 & 1 & 1\\
1 & 2 & 1 & 1 \\
1 & 2 & 1 & 3 \\
1 & 2 & 1 & 4 \\
\end{tabular}
&
$S_{(4, 302)}$
&
\begin{tabular}{cccc}
1 & 2 & 1 & 4\\
1 & 2 & 1 & 4 \\
1 & 2 & 1 & 4 \\
1 & 2 & 1 & 4 \\
\end{tabular}\\
\hline

$S_{(4, 303)}$
&
\begin{tabular}{cccc}
1 & 2 & 1 & 1\\
1 & 2 & 1 & 1 \\
1 & 2 & 1 & 1 \\
1 & 2 & 3 & 4 \\
\end{tabular}
&
$S_{(4, 304)}$
&
\begin{tabular}{cccc}
1 & 2 & 1 & 1\\
1 & 2 & 1 & 1 \\
1 & 2 & 1 & 3 \\
1 & 2 & 3 & 4 \\
\end{tabular}
&
$S_{(4, 305)}$
&
\begin{tabular}{cccc}
1 & 2 & 1 & 1\\
1 & 2 & 1 & 1 \\
1 & 2 & 3 & 1 \\
1 & 2 & 1 & 4 \\
\end{tabular}\\
\hline

$S_{(4, 306)}$
&
\begin{tabular}{cccc}
1 & 2 & 1 & 4\\
1 & 2 & 1 & 4 \\
1 & 2 & 3 & 4 \\
1 & 2 & 1 & 4 \\
\end{tabular}
&
$S_{(4, 307)}$
&
\begin{tabular}{cccc}
1 & 2 & 1 & 1\\
1 & 2 & 1 & 1 \\
1 & 2 & 3 & 4 \\
1 & 2 & 4 & 3 \\
\end{tabular}
&
$S_{(4, 308)}$
&
\begin{tabular}{cccc}
1 & 2 & 3 & 4 \\
1 & 2 & 3 & 4 \\
1 & 2 & 3 & 4 \\
1 & 2 & 3 & 4 \\
\end{tabular}\\
\hline

$S_{(4, 309)}$
&
\begin{tabular}{cccc}
1 & 1 & 1 & 1\\
1 & 1 & 1 & 1 \\
1 & 1 & 1 & 1 \\
4 & 4 & 4 & 4 \\
\end{tabular}
&
$S_{(4, 310)}$
&
\begin{tabular}{cccc}
1 & 1 & 1 & 4 \\
1 & 1 & 1 & 4 \\
1 & 1 & 1 & 4 \\
4 & 4 & 4 & 4 \\
\end{tabular}
&
$S_{(4, 311)}$
&
\begin{tabular}{cccc}
1 & 1 & 1 & 1 \\
1 & 1 & 1 & 1 \\
1 & 1 & 3 & 1 \\
4 & 4 & 4 & 4 \\
\end{tabular}\\
\hline
\end{tabular}
\end{table}
\begin{table}[htbp]
\begin{tabular}{cccccc}
\hline
$S_{(4, 312)}$
&
\begin{tabular}{cccc}
1 & 1 & 1 & 4\\
1 & 1 & 1 & 4 \\
1 & 1 & 3 & 4 \\
4 & 4 & 4 & 4 \\
\end{tabular}
&
$S_{(4, 313)}$
&
\begin{tabular}{cccc}
1 & 1 & 1 & 1 \\
1 & 2 & 1 & 1 \\
1 & 1 & 1 & 1 \\
4 & 4 & 4 & 4 \\
\end{tabular}
&
$S_{(4, 314)}$
&
\begin{tabular}{cccc}
1 & 1 & 1 & 4 \\
1 & 2 & 1 & 4 \\
1 & 1 & 1 & 4 \\
4 & 4 & 4 & 4 \\
\end{tabular}\\
\hline

$S_{(4, 315)}$
&
\begin{tabular}{cccc}
1 & 1 & 1 & 1\\
1 & 2 & 1 & 1 \\
1 & 1 & 3 & 1 \\
4 & 4 & 4 & 4 \\
\end{tabular}
&
$S_{(4, 316)}$
&
\begin{tabular}{cccc}
1 & 1 & 1 & 4 \\
1 & 2 & 1 & 4 \\
1 & 1 & 3 & 4 \\
4 & 4 & 4 & 4 \\
\end{tabular}
&
$S_{(4, 317)}$
&
\begin{tabular}{cccc}
1 & 2 & 1 & 4 \\
1 & 2 & 1 & 4 \\
1 & 2 & 1 & 4 \\
4 & 4 & 4 & 4 \\
\end{tabular}\\
\hline

$S_{(4, 318)}$
&
\begin{tabular}{cccc}
1 & 2 & 1 & 4\\
1 & 2 & 1 & 4 \\
1 & 2 & 3 & 4 \\
4 & 4 & 4 & 4 \\
\end{tabular}
&
$S_{(4, 319)}$
&
\begin{tabular}{cccc}
1 & 1 & 1 & 1 \\
1 & 1 & 1 & 1 \\
3 & 3 & 3 & 3 \\
4 & 4 & 4 & 4 \\
\end{tabular}
&
$S_{(4, 320)}$
&
\begin{tabular}{cccc}
1 & 1 & 1 & 1 \\
1 & 2 & 1 & 1 \\
3 & 3 & 3 & 3 \\
4 & 4 & 4 & 4 \\
\end{tabular}\\
\hline

$S_{(4, 321)}$
&
\begin{tabular}{cccc}
1 & 1 & 1 & 1\\
2 & 2 & 2 & 2 \\
1 & 1 & 1 & 1 \\
1 & 1 & 1 & 1 \\
\end{tabular}
&
$S_{(4, 322)}$
&
\begin{tabular}{cccc}
1 & 1 & 1 & 1 \\
2 & 2 & 2 & 2 \\
1 & 1 & 1 & 1 \\
1 & 1 & 1 & 3 \\
\end{tabular}
&
$S_{(4, 323)}$
&
\begin{tabular}{cccc}
1 & 1 & 1 & 1 \\
2 & 2 & 2 & 2 \\
1 & 1 & 1 & 1 \\
1 & 1 & 1 & 4 \\
\end{tabular}\\
\hline

$S_{(4, 324)}$
&
\begin{tabular}{cccc}
1 & 1 & 1 & 1\\
2 & 2 & 2 & 2 \\
1 & 1 & 1 & 3 \\
1 & 1 & 1 & 4 \\
\end{tabular}
&
$S_{(4, 325)}$
&
\begin{tabular}{cccc}
1 & 1 & 1 & 1 \\
2 & 2 & 2 & 2 \\
1 & 1 & 1 & 1 \\
1 & 1 & 3 & 4 \\
\end{tabular}
&
$S_{(4, 326)}$
&
\begin{tabular}{cccc}
1 & 1 & 1 & 1 \\
2 & 2 & 2 & 2 \\
1 & 1 & 1 & 3 \\
1 & 1 & 3 & 4 \\
\end{tabular}\\
\hline

$S_{(4, 327)}$
&
\begin{tabular}{cccc}
1 & 1 & 1 & 1 \\
2 & 2 & 2 & 2 \\
1 & 1 & 3 & 1 \\
1 & 1 & 1 & 4 \\
\end{tabular}
&
$S_{(4, 328)}$
&
\begin{tabular}{cccc}
1 & 1 & 1 & 1 \\
2 & 2 & 2 & 2 \\
1 & 1 & 3 & 4 \\
1 & 1 & 4 & 3 \\
\end{tabular}
&
$S_{(4, 329)}$
&
\begin{tabular}{cccc}
1 & 2 & 1 & 1 \\
2 & 2 & 2 & 2 \\
1 & 2 & 1 & 1 \\
1 & 2 & 1 & 1 \\
\end{tabular}\\
\hline

$S_{(4, 330)}$
&
\begin{tabular}{cccc}
1 & 2 & 1 & 1 \\
2 & 2 & 2 & 2 \\
1 & 2 & 1 & 1 \\
1 & 2 & 1 & 3 \\
\end{tabular}
&
$S_{(4, 331)}$
&
\begin{tabular}{cccc}
1 & 2 & 1 & 1 \\
2 & 2 & 2 & 2 \\
1 & 2 & 1 & 1 \\
1 & 2 & 1 & 4 \\
\end{tabular}
&
$S_{(4, 332)}$
&
\begin{tabular}{cccc}
1 & 2 & 1 & 1 \\
2 & 2 & 2 & 2 \\
1 & 2 & 1 & 3 \\
1 & 2 & 1 & 4 \\
\end{tabular}\\
\hline

$S_{(4, 333)}$
&
\begin{tabular}{cccc}
1 & 2 & 1 & 4 \\
2 & 2 & 2 & 2 \\
1 & 2 & 1 & 4 \\
1 & 2 & 1 & 4 \\
\end{tabular}
&
$S_{(4, 334)}$
&
\begin{tabular}{cccc}
1 & 2 & 1 & 1 \\
2 & 2 & 2 & 2 \\
1 & 2 & 1 & 1 \\
1 & 2 & 3 & 4 \\
\end{tabular}
&
$S_{(4, 335)}$
&
\begin{tabular}{cccc}
1 & 2 & 1 & 1 \\
2 & 2 & 2 & 2 \\
1 & 2 & 1 & 3 \\
1 & 2 & 3 & 4 \\
\end{tabular}\\
\hline

$S_{(4, 336)}$
&
\begin{tabular}{cccc}
1 & 2 & 1 & 1 \\
2 & 2 & 2 & 2 \\
1 & 2 & 3 & 1 \\
1 & 2 & 1 & 4 \\
\end{tabular}
&
$S_{(4, 337)}$
&
\begin{tabular}{cccc}
1 & 2 & 1 & 4 \\
2 & 2 & 2 & 2 \\
1 & 2 & 3 & 4 \\
1 & 2 & 1 & 4 \\
\end{tabular}
&
$S_{(4, 338)}$
&
\begin{tabular}{cccc}
1 & 2 & 1 & 1 \\
2 & 2 & 2 & 2 \\
1 & 2 & 3 & 4 \\
1 & 2 & 4 & 3 \\
\end{tabular}\\
\hline

$S_{(4, 339)}$
&
\begin{tabular}{cccc}
1 & 2 & 3 & 4 \\
2 & 2 & 2 & 2 \\
1 & 2 & 3 & 4 \\
1 & 2 & 3 & 4 \\
\end{tabular}
&
$S_{(4, 340)}$
&
\begin{tabular}{cccc}
1 & 1 & 1 & 1 \\
2 & 2 & 2 & 2 \\
1 & 1 & 1 & 1 \\
4 & 4 & 4 & 4 \\
\end{tabular}
&
$S_{(4, 341)}$
&
\begin{tabular}{cccc}
1 & 1 & 1 & 4 \\
2 & 2 & 2 & 4 \\
1 & 1 & 1 & 4 \\
4 & 4 & 4 & 4 \\
\end{tabular}\\
\hline

$S_{(4, 342)}$
&
\begin{tabular}{cccc}
1 & 1 & 1 & 1 \\
2 & 2 & 2 & 2 \\
1 & 1 & 3 & 1 \\
4 & 4 & 4 & 4 \\
\end{tabular}
&
$S_{(4, 343)}$
&
\begin{tabular}{cccc}
1 & 1 & 1 & 4 \\
2 & 2 & 2 & 4 \\
1 & 1 & 3 & 4 \\
4 & 4 & 4 & 4 \\
\end{tabular}
&
$S_{(4, 344)}$
&
\begin{tabular}{cccc}
1 & 2 & 1 & 1 \\
2 & 2 & 2 & 2 \\
1 & 2 & 1 & 1 \\
4 & 2 & 4 & 4 \\
\end{tabular}\\
\hline

$S_{(4, 345)}$
&
\begin{tabular}{cccc}
1 & 2 & 1 & 4 \\
2 & 2 & 2 & 2 \\
1 & 2 & 1 & 4 \\
4 & 2 & 4 & 4 \\
\end{tabular}
&
$S_{(4, 346)}$
&
\begin{tabular}{cccc}
1 & 2 & 1 & 4 \\
2 & 2 & 2 & 4 \\
1 & 2 & 1 & 4 \\
4 & 2 & 4 & 4 \\
\end{tabular}
&
$S_{(4, 347)}$
&
\begin{tabular}{cccc}
1 & 2 & 1 & 1 \\
2 & 2 & 2 & 2 \\
1 & 2 & 3 & 1 \\
4 & 2 & 4 & 4 \\
\end{tabular}\\
\hline

$S_{(4, 348)}$
&
\begin{tabular}{cccc}
1 & 2 & 1 & 4 \\
2 & 2 & 2 & 2 \\
1 & 2 & 3 & 4 \\
4 & 2 & 4 & 4 \\
\end{tabular}
&
$S_{(4, 349)}$
&
\begin{tabular}{cccc}
1 & 2 & 1 & 4 \\
2 & 2 & 2 & 4 \\
1 & 2 & 3 & 4 \\
4 & 2 & 4 & 4 \\
\end{tabular}
&
$S_{(4, 350)}$
&
\begin{tabular}{cccc}
1 & 2 & 1 & 4 \\
2 & 2 & 2 & 2 \\
1 & 2 & 1 & 4 \\
4 & 4 & 4 & 4 \\
\end{tabular}\\
\hline
\end{tabular}
\end{table}
\begin{table}[htbp]
\begin{tabular}{cccccc}
\hline
$S_{(4, 351)}$
&
\begin{tabular}{cccc}
1 & 2 & 1 & 4\\
2 & 2 & 2 & 4 \\
1 & 2 & 1 & 4 \\
4 & 4 & 4 & 4 \\
\end{tabular}
&
$S_{(4, 352)}$
&
\begin{tabular}{cccc}
1 & 2 & 1 & 4 \\
2 & 2 & 2 & 2 \\
1 & 2 & 3 & 4 \\
4 & 4 & 4 & 4 \\
\end{tabular}
&
$S_{(4, 353)}$
&
\begin{tabular}{cccc}
1 & 2 & 1 & 4 \\
2 & 2 & 2 & 4 \\
1 & 2 & 3 & 4 \\
4 & 4 & 4 & 4 \\
\end{tabular}\\
\hline

$S_{(4, 354)}$
&
\begin{tabular}{cccc}
1 & 1 & 1 & 1\\
2 & 2 & 2 & 2 \\
3 & 3 & 3 & 3 \\
4 & 4 & 4 & 4 \\
\end{tabular}
&
$S_{(4, 355)}$
&
\begin{tabular}{cccc}
1 & 2 & 1 & 1 \\
2 & 2 & 2 & 2 \\
3 & 2 & 3 & 3 \\
4 & 2 & 4 & 4 \\
\end{tabular}
&
$S_{(4, 356)}$
&
\begin{tabular}{cccc}
2 & 2 & 1 & 2 \\
2 & 2 & 2 & 2 \\
1 & 2 & 3 & 1 \\
2 & 2 & 1 & 2 \\
\end{tabular}\\
\hline

$S_{(4, 357)}$
&
\begin{tabular}{cccc}
2 & 2 & 1 & 2\\
2 & 2 & 2 & 2 \\
1 & 2 & 3 & 4 \\
2 & 2 & 1 & 2 \\
\end{tabular}
&
$S_{(4, 358)}$
&
\begin{tabular}{cccc}
2 & 2 & 1 & 2\\
2 & 2 & 2 & 2 \\
1 & 2 & 3 & 1 \\
2 & 2 & 4 & 2 \\
\end{tabular}
&
$S_{(4, 359)}$
&
\begin{tabular}{cccc}
2 & 2 & 1 & 2\\
2 & 2 & 2 & 2 \\
1 & 2 & 3 & 4 \\
2 & 2 & 4 & 2 \\
\end{tabular}\\
\hline

$S_{(4, 360)}$
&
\begin{tabular}{cccc}
2 & 2 & 2 & 1\\
2 & 2 & 2 & 2 \\
1 & 2 & 3 & 1 \\
2 & 2 & 2 & 4\\
\end{tabular}
&
$S_{(4, 361)}$
&
\begin{tabular}{cccc}
2 & 2 & 2 & 2 \\
2 & 2 & 2 & 2 \\
1 & 2 & 3 & 1 \\
2 & 2 & 2 & 2 \\
\end{tabular}
&
$S_{(4, 362)}$
&
\begin{tabular}{cccc}
2 & 2 & 2 & 2 \\
2 & 2 & 2 & 2 \\
1 & 2 & 3 & 4 \\
2 & 2 & 2 & 2 \\
\end{tabular}\\
\hline

$S_{(4, 363)}$
&
\begin{tabular}{cccc}
2 & 2 & 2 & 4 \\
2 & 2 & 2 & 4 \\
1 & 2 & 3 & 4 \\
2 & 2 & 2 & 4 \\
\end{tabular}
&
$S_{(4, 364)}$
&
\begin{tabular}{cccc}
2 & 2 & 1 & 2 \\
2 & 2 & 2 & 2 \\
2 & 2 & 3 & 2 \\
2 & 2 & 1 & 2 \\
\end{tabular}
&
$S_{(4, 365)}$
&
\begin{tabular}{cccc}
2 & 2 & 1 & 2 \\
2 & 2 & 2 & 2 \\
2 & 2 & 3 & 2 \\
2 & 2 & 4 & 2 \\
\end{tabular}\\
\hline

$S_{(4, 366)}$
&
\begin{tabular}{cccc}
2 & 2 & 2 & 2 \\
2 & 2 & 2 & 2 \\
2 & 2 & 1 & 2 \\
2 & 2 & 2 & 1 \\
\end{tabular}
&
$S_{(4, 367)}$
&
\begin{tabular}{cccc}
2 & 2 & 2 & 2 \\
2 & 2 & 2 & 2 \\
2 & 2 & 1 & 2 \\
2 & 2 & 2 & 2 \\
\end{tabular}
&
$S_{(4, 368)}$
&
\begin{tabular}{cccc}
2 & 2 & 2 & 2 \\
2 & 2 & 2 & 2 \\
2 & 2 & 1 & 2 \\
2 & 2 & 2 & 4 \\
\end{tabular}\\
\hline

$S_{(4, 369)}$
&
\begin{tabular}{cccc}
2 & 2 & 2 & 2 \\
2 & 2 & 2 & 2 \\
2 & 2 & 2 & 1 \\
2 & 2 & 1 & 2 \\
\end{tabular}
&
$S_{(4, 370)}$
&
\begin{tabular}{cccc}
2 & 2 & 2 & 2 \\
2 & 2 & 2 & 2 \\
2 & 2 & 2 & 2 \\
2 & 2 & 1 & 2 \\
\end{tabular}
&
$S_{(4, 371)}$
&
\begin{tabular}{cccc}
2 & 2 & 2 & 2 \\
2 & 2 & 2 & 2 \\
2 & 2 & 2 & 2 \\
2 & 2 & 2 & 2 \\
\end{tabular}\\

\hline
$S_{(4, 372)}$
&
\begin{tabular}{cccc}
2 & 2 & 2 & 2 \\
2 & 2 & 2 & 2 \\
2 & 2 & 2 & 2 \\
2 & 2 & 2 & 3 \\
\end{tabular}
&
$S_{(4, 373)}$
&
\begin{tabular}{cccc}
2 & 2 & 2 & 2 \\
2 & 2 & 2 & 2 \\
2 & 2 & 2 & 2 \\
2 & 2 & 2 & 4 \\
\end{tabular}
&
$S_{(4, 374)}$
&
\begin{tabular}{cccc}
2 & 2 & 2 & 4 \\
2 & 2 & 2 & 4 \\
2 & 2 & 2 & 4 \\
2 & 2 & 2 & 4 \\
\end{tabular}\\
\hline

$S_{(4, 375)}$
&
\begin{tabular}{cccc}
2 & 2 & 2 & 2 \\
2 & 2 & 2 & 2 \\
2 & 2 & 3 & 2 \\
2 & 2 & 2 & 4 \\
\end{tabular}
&
$S_{(4, 376)}$
&
\begin{tabular}{cccc}
2 & 2 & 1 & 2 \\
2 & 2 & 2 & 2 \\
2 & 2 & 3 & 2 \\
4 & 4 & 4 & 4 \\
\end{tabular}
&
$S_{(4, 377)}$
&
\begin{tabular}{cccc}
2 & 2 & 2 & 2 \\
2 & 2 & 2 & 2 \\
2 & 2 & 2 & 2 \\
4 & 4 & 4 & 4 \\
\end{tabular}\\
\hline

$S_{(4, 378)}$
&
\begin{tabular}{cccc}
2 & 2 & 2 & 4 \\
2 & 2 & 2 & 4 \\
2 & 2 & 2 & 4 \\
4 & 4 & 4 & 4 \\
\end{tabular}
&
$S_{(4, 379)}$
&
\begin{tabular}{cccc}
3 & 1 & 3 & 3 \\
1 & 2 & 3 & 1 \\
3 & 3 & 3 & 3 \\
3 & 1 & 3 & 3 \\
\end{tabular}
&
$S_{(4, 380)}$
&
\begin{tabular}{cccc}
3 & 3 & 3 & 3 \\
1 & 2 & 3 & 1 \\
3 & 3 & 3 & 3 \\
3 & 3 & 3 & 3 \\
\end{tabular}\\
\hline

$S_{(4, 381)}$
&
\begin{tabular}{cccc}
3 & 2 & 3 & 3 \\
2 & 2 & 2 & 2 \\
3 & 2 & 3 & 3 \\
3 & 2 & 3 & 3 \\
\end{tabular}
&
$S_{(4, 382)}$
&
\begin{tabular}{cccc}
3 & 3 & 3 & 3 \\
2 & 2 & 2 & 2 \\
3 & 3 & 3 & 3 \\
3 & 3 & 3 & 3 \\
\end{tabular}
&
$S_{(4, 383)}$
&
\begin{tabular}{cccc}
3 & 1 & 3 & 3 \\
3 & 2 & 3 & 3 \\
3 & 3 & 3 & 3 \\
3 & 1 & 3 & 3 \\
\end{tabular}\\
\hline

$S_{(4, 384)}$
&
\begin{tabular}{cccc}
3 & 2 & 3 & 3 \\
3 & 2 & 3 & 3 \\
3 & 2 & 3 & 3 \\
3 & 2 & 3 & 3 \\
\end{tabular}
&
$S_{(4, 385)}$
&
\begin{tabular}{cccc}
3 & 3 & 3 & 3 \\
3 & 1 & 3 & 3 \\
3 & 3 & 3 & 3 \\
3 & 3 & 3 & 3 \\
\end{tabular}
&
$S_{(4, 386)}$
&
\begin{tabular}{cccc}
3 & 3 & 3 & 3 \\
3 & 2 & 3 & 3 \\
3 & 3 & 3 & 3 \\
3 & 3 & 3 & 3 \\
\end{tabular}\\
\hline

$S_{(4, 387)}$
&
\begin{tabular}{cccc}
3 & 3 & 3 & 3 \\
3 & 3 & 3 & 3 \\
3 & 3 & 3 & 3 \\
3 & 3 & 3 & 3 \\
\end{tabular}\\
\hline
\end{tabular}
\end{table}

\section{The finite basis problem for specific 4-element ai-semirings}
In this section we solve the finite basis problem for some $4$-element ai-semirings by some known results.
\begin{pro}\label{pro27601}
The following ai-semirings are finitely based\up: $S_{(4, 276)}$,  $S_{(4, 278)}$, $S_{(4, 280)}$, $S_{(4, 283)}$, $S_{(4, 284)}$, $S_{(4, 286)}$, $S_{(4, 287)}$, $S_{(4, 289)}$, $S_{(4, 291)}$, $S_{(4, 298)}$, $S_{(4, 300)}$, $S_{(4, 302)}$, $S_{(4, 309)}$, $S_{(4, 310)}$, $S_{(4, 311)}$, $S_{(4, 312)}$, $S_{(4, 313)}$, $S_{(4, 314)}$, $S_{(4, 319)}$, $S_{(4, 321)}$, $S_{(4, 323)}$, $S_{(4, 329)}$, $S_{(4, 331)}$, $S_{(4, 333)}$, $S_{(4, 340)}$, $S_{(4, 344)}$, $S_{(4, 345)}$, $S_{(4, 371)}$, $S_{(4, 381)}$, $S_{(4, 382)}$, $S_{(4, 384)}$, $S_{(4, 386)}$ and $S_{(4, 387)}$.
\end{pro}
\begin{proof}
It is easy to check that every semiring in Proposition $\ref{pro27601}$
satisfies the equational basis (see \cite[Theorem 2.1]{sr})
of the variety generated by all ai-semirings of order two.
By the main result of \cite{sr} we have that these algebras are all finitely based.
\end{proof}

\begin{pro}\label{pro29401}
The following ai-semirings are finitely based\up: $S_{(4, 294)}$, $S_{(4, 295)}$, $S_{(4, 296)}$, $S_{(4, 297)}$,
$S_{(4, 305)}$, $S_{(4, 306)}$, $S_{(4, 307)}$, $S_{(4, 308)}$, $S_{(4, 315)}$, $S_{(4, 316)}$, $S_{(4, 318)}$, $S_{(4, 320)}$,
$S_{(4, 327)}$, $S_{(4, 328)}$, $S_{(4, 336)}$, $S_{(4, 337)}$, $S_{(4, 338)}$, $S_{(4, 339)}$, $S_{(4, 342)}$, $S_{(4, 343)}$, $S_{(4, 347)}$, $S_{(4, 348)}$, $S_{(4, 349)}$, $S_{(4, 352)}$, $S_{(4, 353)}$, $S_{(4, 354)}$ and $S_{(4, 355)}$.
\end{pro}
\begin{proof}
It is easy to see that every semiring in Proposition $\ref{pro29401}$ satisfies $x^3 \approx x$. From
the main result of \cite{rzw} we deduce that these semirings are all finitely based.
\end{proof}
\begin{pro}\label{pro31701}
The following ai-semirings are finitely based\up:
$S_{(4, 317)}$, $S_{(4, 341)}$, $S_{(4, 346)}$, $S_{(4, 350)}$ and $S_{(4, 351)}$.
\end{pro}
\begin{proof}
It is easy to see that every semiring in Proposition $\ref{pro31701}$ satisfies
the equational basis (see \cite{rz}) of the variety that is the join of the ai-semiring variety defined by $x^2\approx x$
and the ai-semiring variety defined by $x_1x_2 \approx y_1y_2$.
By the main result of \cite{rz} we obtain that these algebras are all finitely based.
\end{proof}
\begin{pro}
The following ai-semirings are finitely based\up: $S_{(4, 332)}$, $S_{(4, 334)}$, $S_{(4, 356)}$, $S_{(4, 361)}$, $S_{(4, 364)}$, $S_{(4, 367)}$, $S_{(4, 373)}$, $S_{(4, 374)}$, $S_{(4, 377)}$ and $S_{(4, 378)}$.
\end{pro}
\begin{proof}
It is easy to verify that $S_{(4, 332)}$ is isomorphic to a subdirect product of $S_{30}$ and $S_{6}$,
and that $S_{30}$ satisfies the finite equational basis (see \cite[Proposition 3]{zrc}) of $S_{6}$.
Thus $\mathsf{V}(S_{(4, 332)}) = \mathsf{V}(S_{6})$
and so $S_{(4,332)}$ is finitely based. One can use the same approach to prove that
\[\mathsf{V}(S_{(4, 334)}) = \mathsf{V}(S_{4}),
\mathsf{V}(S_{(4, 356)}) = \mathsf{V}(S_{53}),
\mathsf{V}(S_{(4, 361)}) = \mathsf{V}(S_{57}),
\]
\[
\mathsf{V}(S_{(4, 364)}) = \mathsf{V}(S_{54}),
\mathsf{V}(S_{(4, 367)}) = \mathsf{V}(S_{59}),
\mathsf{V}(S_{(4, 373)}) = \mathsf{V}(S_{60}),
\]
\[
\mathsf{V}(S_{(4, 374)}) = \mathsf{V}(S_{56}),
\mathsf{V}(S_{(4, 377)}) = \mathsf{V}(S_{58}),
\mathsf{V}(S_{(4, 378)}) = \mathsf{V}(S_{55}).
\]
By the main result of \cite{zrc} we deduce that
$S_{(4, 334)}$, $S_{(4, 356)}$, $S_{(4, 361)}$, $S_{(4, 364)}$, $S_{(4, 367)}$, $S_{(4, 373)}$, $S_{(4, 374)}$,
$S_{(4, 377)}$ and $S_{(4, 378)}$ are all finitely based.
\end{proof}
\begin{pro}\label{pro29301}
The following ai-semirings are nonfinitely based\up: $S_{(4, 282)}$, $S_{(4, 293)}$, $S_{(4, 304)}$, $S_{(4, 326)}$, $S_{(4, 335)}$
and $S_{(4, 359)}$.
\end{pro}
\begin{proof}
Let $k$ be a number in $\{293, 304, 326, 335\}$.
It is easy to verify that $S_{(4, k)}$ contains a copy of $S_7$
and that the set of all noncyclic elements of $S_{(4, k)}$ forms an order ideal.
From \cite[Proposition 2.2]{rlzc} we deduce that $S_{(4, k)}$ is nonfinitely based.
On the other hand, $S_{(4, 282)}$ and $S_{(4, 359)}$ are both nonfinitely based
based by \cite[Corollary 2.7]{gmrz} and \cite[Corollary 2.9]{gmrz}, respectively.
\end{proof}

\section{Equational bases of some 4-element ai-semirings related to $S_2$}
In this section we study the finite basis problem for some 4-element ai-semirings that relate to $S_2$.
The following result, which is \cite[Lemma 4.1]{rlzc}, provides a solution of the equational problem for $S_2$.
\begin{lem}\label{lem201}
Let $\bu\approx \bu+\bq$ be an ai-semiring identity such that
$\bu=\bu_1+\bu_2+\cdots+\bu_n$ and $\bu_i, \bq \in X^+$, $1\leq i \leq n$.
Then $\bu\approx \bu+\bq$ is satisfied by $S_2$ if and only if $\bu$ and $\bq$ satisfy one of the following conditions\up:
\begin{itemize}
\item[$(1)$] $\ell(\bu_i)\geq 3$ for some $\bu_i \in \bu$\up;

\item[$(2)$] $c(L_1(\bu))\cap c(L_2(\bu)) \neq\emptyset$\up;

\item[$(3)$] $\ell(\bu_i)\leq 2$ for all $\bu_i \in \bu$, $c(L_1(\bu))\cap c(L_2(\bu))=\emptyset$ and $\ell(\bq)\leq 2$.
             If $\ell(\bq)=1$, then $\bu\approx \bu+\bq$ is trivial.
             If $\ell(\bq)=2$, then $c(\bq)\subseteq c(L_2(\bu))$.
\end{itemize}
\end{lem}

In what follows, we shall use $X_c^+$ to denote the free commutative semigroup over the alphabet $X$.
\begin{pro}\label{pro27701}
$\mathsf{V}(S_{(4, 277)})$ is the ai-semiring variety defined by the identities
\begin{align}
xy  & \approx x^2+y^2; \label{27702} \\
x+x^2& \approx x+x^3; \label{27703} \\
x_1x_2x_3 & \approx x_1x_2x_3+y_1y_2. \label{27704}
\end{align}
\end{pro}
\begin{proof}
It is easy to check that $S_{(4, 277)}$ satisfies the identities \eqref{27702}--\eqref{27704}.
In the remainder it is enough to prove that every ai-semiring identity of $S_{(4, 277)}$
can be derived by \eqref{27702}--\eqref{27704} and the identities defining $\mathbf{AI}$.
Let $\bu \approx \bu+\bq$ be such a nontrivial identity,
where $\bu=\bu_1+\bu_2+\cdots+\bu_n$ and $\bu_i, \bq \in X^+$, $1 \leq i \leq n$.
By the identity \eqref{27702}, one can deduce the identity $xy \approx yx$.
So we may assume that $\bu_i, \bq \in X_c^+$ for $1 \leq i \leq n$.
It is easy to see that $N_2$ is isomorphic to $\{1, 2\}$ and so $N_2$ satisfies  $\bu \approx \bu+\bq$.
This implies that $\ell(\bq)\geq 2$.
Since $S_2$ is isomorphic to $\{1, 3, 4\}$, it follows that $S_2$ satisfies $\bu \approx \bu+\bq$.
By Lemma \ref{lem201} we only need to consider the following three cases.

\textbf{Case 1.} $\ell(\bu_j)\geq 3$ for some $\bu_j\in \bu$. Then
\[
\bu \approx \bu+\bu_j \stackrel{\eqref{27704}}\approx \bu+\bu_j+\bq \approx \bu+\bq.
\]

\textbf{Case 2.} $c(L_1(\bu))\cap c(L_2(\bu))\neq\emptyset$.
Take $x$ in $c(L_1(\bu))\cap c(L_2(\bu))$.
Then $xy\in L_2(\bu)$ for some $y\in X$.
Now we have
\[
\bu \approx \bu+x+xy \stackrel{\eqref{27702}}\approx \bu+x+x^2+y^2 \stackrel{\eqref{27703}}\approx \bu+x+x^3+y^2.
\]
The remaining steps are similar to Case 1.

\textbf{Case 3.} $\ell(\bu_j)\leq 2$ for all $\bu_j \in \bu $,
$c(L_1(\bu))\cap c(L_2(\bu))=\emptyset$ and $\ell(\bq)=2$.
Then $c(\bq)\subseteq c(L_2(\bu))$.
If we write $\bq=xy$, then $xx_1, yy_1 \in L_2(\bu)$ for some $x_1, y_1 \in X$.
Furthermore, we have
\begin{align*}
\bu
&\approx \bu+xx_1+yy_1\\
&\approx \bu+x^2+x_1^2+y^2+y_1^2 &&(\text{by}~\eqref{27702})\\
&\approx \bu+x_1^2+y_1^2+x^2+y^2 \\
&\approx \bu+x_1y_1+xy &&(\text{by}~\eqref{27702})\\
&\approx \bu+x_1y_1+\bq.
\end{align*}
This derives the identity $\bu\approx \bu+\bq$.
\end{proof}

\begin{remark}
It is a routine matter to verify that $S_{(4, 277)}$ is isomorphic to a subdirect product of $S_{2}$ and $S_{18}$.
So $\mathsf{V}(S_{(4, 277)}) = \mathsf{V}(S_{2}, S_{18})$.
On the other hand, it is easy
to see that both $N_2$ and $T_2$ can be embedded into $S_{18}$.
Also, $S_{18}$ satisfies the equational basis of $\mathsf{V}(N_2, T_2)$ that can be found in \cite{sr}.
Thus $\mathsf{V}(S_{18})=\mathsf{V}(N_2, T_2)$ and so $\mathsf{V}(S_{(4, 277)})=\mathsf{V}(S_{2}, N_2, T_2)$.
Since $T_2$ can be embedded into $S_{2}$, we therefore have
\[
\mathsf{V}(S_{(4, 277)})=\mathsf{V}(S_{2}, N_2).
\]
\end{remark}

\begin{pro}
$S_{(4, 288)}$ is finitely based.
\end{pro}
\begin{proof}
It is easy to check that $S_{(4, 288)}$ is isomorphic to a subdirect product of $S_{2}$ and $S_{19}$.
This implies that $\mathsf{V}(S_{(4, 288)}) = \mathsf{V}(S_{2}, S_{19})$.
On the other hand, it is easy
to see that both $D_2$ and $T_2$ can be embedded into $S_{19}$.
Also, $S_{19}$ satisfies the equational basis of $\mathsf{V}(D_2, T_2)$ that can be found in \cite{sr}.
Hence $\mathsf{V}(S_{19})=\mathsf{V}(D_2, T_2)$ and so $\mathsf{V}(S_{(4, 288)})=\mathsf{V}(S_{2}, D_2, T_2)$.
Since $T_2$ can be embedded into $S_{2}$, we therefore obtain
\[
\mathsf{V}(S_{(4, 288)})=\mathsf{V}(S_{2}, D_2).
\]
By \cite[Remark 4.8]{rlzc} it follows that $\mathsf{V}(S_{(4, 42)})$ is equal to $\mathsf{V}(S_{2}, D_2)$ and is finitely based.
Thus $\mathsf{V}(S_{(4, 288)})=\mathsf{V}(S_{(4, 42)})$ and so $S_{(4, 288)}$ is also finitely based.
\end{proof}

\begin{pro}
$S_{(4, 299)}$ and $S_{(4, 322)}$ are both finitely based.
\end{pro}
\begin{proof}
It is easy to verify that $S_{(4, 299)}$ is isomorphic to a subdirect product of $S_{2}$ and $S_{27}$.
So $\mathsf{V}(S_{(4, 299)}) = \mathsf{V}(S_{2}, S_{27})$.
On the other hand, it is easy
to see that both $R_2$ and $T_2$ can be embedded into $S_{27}$.
Also, $S_{27}$ satisfies the equational basis of $\mathsf{V}(R_2, T_2)$ that can be found in \cite{sr}.
Thus $\mathsf{V}(S_{27})=\mathsf{V}(R_2, T_2)$ and so $\mathsf{V}(S_{(4, 299)})=\mathsf{V}(S_{2}, R_2, T_2)$.
Since $T_2$ can be embedded into $S_{2}$, we therefore have
\[
\mathsf{V}(S_{(4, 299)})=\mathsf{V}(S_{2}, R_2).
\]
From \cite[Remark 4.5]{rlzc} and \cite[Corollary 4.6]{rlzc}
we know that $\mathsf{V}(S_{(4, 16)})$ is equal to $\mathsf{V}(S_{2}, R_2)$ and is finitely based.
Hence $\mathsf{V}(S_{(4, 299)})=\mathsf{V}(S_{(4, 16)})$ and so $S_{(4, 299)}$ is finitely based.
Notice that $S_{(4, 322)}$ and $S_{(4, 299)}$ have dual multiplications.
Thus $S_{(4, 322)}$ is also finitely based.
\end{proof}

\begin{pro}
$S_{(4, 330)}$ is finitely based.
\end{pro}
\begin{proof}
It is a routine matter to verify that $S_{(4, 330)}$ is isomorphic to a subdirect product of $S_{2}$ and $S_{28}$.
So $\mathsf{V}(S_{(4, 330)}) = \mathsf{V}(S_{2}, S_{28})$.
On the other hand, it is easy
to see that both $M_2$ and $T_2$ can be embedded into $S_{28}$.
Also, $S_{28}$ satisfies the equational basis of $\mathsf{V}(M_2, T_2)$ that can be found in \cite{sr}.
Thus $\mathsf{V}(S_{28})=\mathsf{V}(M_2, T_2)$ and so $\mathsf{V}(S_{(4, 330)})=\mathsf{V}(S_{2}, M_2, T_2)$.
Since $T_2$ can be embedded into $S_{2}$, we therefore have
\[
\mathsf{V}(S_{(4, 330)})=\mathsf{V}(S_{2}, M_2).
\]
By \cite[Remark 4.3]{rlzc} we have that $\mathsf{V}(S_{(4, 15)})$ is equal to $\mathsf{V}(S_{2}, M_2)$ and is finitely based.
Hence $\mathsf{V}(S_{(4, 330)})=\mathsf{V}(S_{(4, 15)})$ and so $S_{(4, 330)}$ is finitely based.
\end{proof}

The following result can be found in \cite[Lemma 5.1]{yrzs}.
\begin{lem}\label{lem5901}
Let $\bu\approx \bu+\bq$ be a nontrivial ai-semiring identity such that
$\bu=\bu_1+\bu_2+\cdots+\bu_n$ and $\bu_i, \bq \in X^+$, $1\leq i \leq n$.
If $\bu\approx \bu+\bq$ is satisfied by $S_{59}$, then $\bu$ and $\bq$ satisfy one of the following conditions\up:
\begin{itemize}
\item[$(1)$] $\ell(\bu_i)\geq 3$ for some $\bu_i \in \bu$\up;

\item[$(2)$] $\ell(\bu_i)\leq 2$ for all $\bu_i \in \bu$. Then $\ell(\bq)\leq 2$.
If $\ell(\bq)=1$, then $c(\bq)\subseteq c(\bu)$. If $\ell(\bq)=2$, then $c(\bq)\subseteq c(L_2(\bu))$.
\end{itemize}
\end{lem}
\begin{pro}\label{pro37201}
$\mathsf{V}(S_{(4, 372)})$ is the ai-semiring variety defined by the identities
\begin{align}
xy  & \approx x^2+y^2; \label{37202} \\
x_1x_2x_3 & \approx x_1x_2x_3+y;\label{37203} \\
x+xy& \approx x+xy+y. \label{37204}
\end{align}
\end{pro}
\begin{proof}
It is easy to verify that $S_{(4, 372)}$ satisfies the identities \eqref{37202}--\eqref{37204}.
In the remainder it is enough to prove that every ai-semiring identity of $S_{(4, 372)}$
is derivable from \eqref{37202}--\eqref{37204} and the identities defining $\mathbf{AI}$.
Let $\bu \approx \bu+\bq$ be such a nontrivial identity,
where $\bu=\bu_1+\bu_2+\cdots+\bu_n$ and $\bu_i, \bq \in X^+$, $1 \leq i \leq n$.
By the identity \eqref{37202}, one can derive the identity
\begin{align}
xy & \approx yx. \label{37201}
\end{align}
So we may assume that $\bu_i, \bq \in X_c^+$ for $1 \leq i \leq n$.
It is a routine matter to verify that $S_{(4, 372)}$ is isomorphic to
a subdirect product of $S_{2}$ and $S_{59}$.
So both $S_{2}$ and $S_{59}$ satisfy $\bu \approx \bu+\bq$.
By Lemma \ref{lem5901} we consider the following two cases.

\textbf{Case 1.} $\ell(\bu_j)\geq 3$ for some $\bu_j\in \bu$. Then
\[
\bu \approx \bu+\bu_j \stackrel{(\ref{37203})}\approx \bu+\bu_j+\bq \approx \bu+\bq.
\]

\textbf{Case 2.} $\ell(\bu_j)\leq 2$ for all $\bu_j\in \bu$, and $\ell(\bq)\leq 2$.

\textbf{Subcase 2.1.} $\ell(\bq)=1$. By Lemma \ref{lem5901} $c(\bq)\subseteq  c(\bu)$
and so there exists $\bu_k \in \bu$ such that $c(\bq)\subseteq  c(\bu_k)$.
Since $\bu \approx \bu+\bq$ is nontrivial, it follows from Lemma \ref{lem201}
that $c(L_1(\bu))\cap c(L_2(\bu))\neq\emptyset$ and $\ell(\bu_k)= 2$.
Take $x$ in $c(L_1(\bu))\cap c(L_2(\bu))$.
Then $xy\in L_2(\bu)$ for some $y\in X$.
Now we have
\begin{align*}
\bu
&\approx \bu+x+xy+\bu_k\\
&\approx \bu+x+xy+\bq\bq_1 &&(\text{by}~\eqref{37201})\\
&\approx \bu+x+x^2+y^2+\bq^2+\bq_1^2 &&(\text{by}~\eqref{37202})\\
&\approx \bu+x+x^2+\bq^2+y^2+\bq_1^2 \\
&\approx \bu+x+x\bq+y\bq_1 &&(\text{by}~\eqref{37202})\\
&\approx \bu+x+x\bq+\bq+y\bq_1. &&(\text{by}~\eqref{37204})
\end{align*}
This derives the identity $\bu\approx \bu+\bq$.

\textbf{Subcase 2.2.} $\ell(\bq)=2$. By Lemma \ref{lem5901} we obtain that $c(\bq)\subseteq  c(L_2(\bu))$.
Let us write $\bq=xy$ for some $x, y\in X$. Then $xx_1, yy_1 \in L_2(\bu)$ for some $x_1, y_1 \in X$ and so
\begin{align*}
\bu
&\approx \bu+xx_1+yy_1\\
&\approx \bu+x^2+x_1^2+y^2+y_1^2 &&(\text{by}~(\ref{37202}))\\
&\approx \bu+x_1^2+y_1^2+x^2+y^2 \\
&\approx \bu+x_1y_1+xy &&(\text{by}~(\ref{37202}))\\
&\approx \bu+x_1y_1+\bq.
\end{align*}
This implies the identity $\bu\approx \bu+\bq$.
\end{proof}

\section{Equational bases of some 4-element ai-semirings related to $S_4$}

In this section we provide equational bases for some 4-element ai-semirings that relate to $S_4$.
Let $\bv$ be an ai-semiring term such that $\bv=\bv_1+\bv_2+\cdots+\bv_k$,
where $\bv_i\in X^+$, $1\leq i \leq k$.
Then $p(\bv)$ denotes $p(\bv_1)+p(\bv_2)+\cdots+p(\bv_k)$.
For example, if $\bv=x+yxz+yy+xzy$, then $p(\bv)=yx+y+xz$.
The following result, which is \cite[Lemma 5.1]{rlzc},
provides a solution of the equational problem for $S_4$.

\begin{lem}\label{lem401}
Let $\bu\approx \bu+\bq$ be a nontrivial ai-semiring identity such that
$\bu=\bu_1+\bu_2+\cdots+\bu_n$ and $\bu_i, \bq \in X^+$, $1\leq i \leq n$.
Then $\bu\approx \bu+\bq$ is satisfied by $S_4$ if and only if
$c(\bq)\subseteq  c(\bu)$, $L_{\geq 2}(\bu)\neq \emptyset$,
and
\[
c(p(\bu))\cap t(\bu)=\emptyset \Rightarrow c(p(\bu+\bq))\cap t(\bu+\bq)=\emptyset.
\]
\end{lem}

\begin{pro}\label{pro28101}
$\mathsf{V}(S_{(4, 281)})$ is the ai-semiring variety defined by the identities
\begin{align}
x^2y& \approx xy; \label{28101}\\
x^2y^2&\approx x^2+y^2;\label{28103}\\
x+y^2& \approx x+xy^2;\label{28104}\\
x+yz& \approx x+yz+yx;\label{28105}\\
x^2+yz&\approx x^2+yz+xy.\label{28106}
\end{align}
\end{pro}
\begin{proof}
It is easy to check that $S_{(4, 281)}$ satisfies the identities \eqref{28101}--\eqref{28106}.
Observe that the identity \eqref{28103} implies the identity
\begin{align}
x^2y^2& \approx y^2x^2, \label{25091801}
\end{align}
since the addition is commutative and the right-hand side of \eqref{28103} does not change
when $x$ and $y$ are swapped. Now we have
\[
xyz \stackrel{\eqref{28101}}\approx x^2yz \stackrel{\eqref{28101}}
\approx  x^2y^2z \stackrel{\eqref{25091801}}\approx y^2x^2z  \stackrel{\eqref{28101}}\approx yx^2z \stackrel{\eqref{28101}}\approx yxz.
\]
This derives the identity
\begin{align}
xyz& \approx yxz. \label{28102}
\end{align}
In the remainder we shall show that every ai-semiring identity of $S_{(4, 281)}$
is derivable from \eqref{28101}--\eqref{28106}, \eqref{28102} and the identities defining $\mathbf{AI}$.
Let $\bu \approx \bu+\bq$ be such a nontrivial identity,
where $\bu=\bu_1+\bu_2+\cdots+\bu_n$ and $\bu_i, \bq \in X^+$, $1 \leq i \leq n$.
It is easy to see that $N_2$ is isomorphic to $\{1, 2\}$ and so $N_2$ satisfies  $\bu \approx \bu+\bq$.
This implies that $\ell(\bq)\geq 2$.
Since $S_4$ is isomorphic to $\{1, 3, 4\}$, it follows that $S_4$ satisfies $\bu \approx \bu+\bq$.
By Lemma \ref{lem401} we have that $c(\bq)\subseteq c(\bu)$ and $L_{\geq 2}(\bu)\neq \emptyset$.

\textbf{Case 1.} $c(p(\bu))\cap t(\bu)=\emptyset$. Then by Lemma \ref{lem401} $c(p(\bu+\bq))\cap t(\bu+\bq)=\emptyset$.
This implies that $m(t(\bq), \bq)=1$ and so $t(\bq)\notin c(p(\bq))$. By the identities (\ref{28101}) and (\ref{28102}) we derive
\begin{equation}\label{id25032501}
\bq \approx x_1^2x_2^2\cdots x_k^2t(\bq),
\end{equation}
where $c(p(\bq))=\{x_1, x_2,\ldots, x_k\}$.
We may assume that $\{x_1, x_2,\ldots, x_r\}$ and $\{y_1, y_2,$ $\ldots, y_s\}$
are two disjoint subsets of $c(\bu)$ such that
\[
c(\bu)=\{x_1, x_2,\ldots, x_r\}\cup \{y_1, y_2,\ldots, y_s\}
\]
and
\[
t(\bu)=\{y_1, y_2,\ldots, y_s\},
\]
where $y_1=t(\bq)$. Now we have
\begin{align*}
\bu
&\approx \bu+x_{k+1}^2x_{k+2}^2 \cdots x_r^2x_1^2x_2^2\cdots x_k^2t(\bq) &&(\text{by}~\eqref{28101}, \eqref{28105}, \eqref{28102}) \\
&\approx \bu+(x_{k+1}^2x_{k+2}^2 \cdots x_r^2+x_1^2x_2^2\cdots x_k^2)t(\bq) &&(\text{by}~\eqref{28103})\\
&\approx \bu+x_{k+1}^2x_{k+2}^2 \cdots x_r^2t(\bq)+x_1^2x_2^2\cdots x_k^2t(\bq)\\
&\approx \bu+x_{k+1}^2x_{k+2}^2 \cdots x_r^2t(\bq)+\bq.    &&(\text{by}~\eqref{id25032501})
\end{align*}
This implies the identity $\bu \approx \bu+\bq$.

\textbf{Case 2.} $c(p(\bu))\cap t(\bu)\neq\emptyset$.
Then $t(\bu_{i_1})\in c(p(\bu_{i_2}))$ for some $\bu_{i_1}, \bu_{i_2} \in \bu$.
Assume that $c(\bu)=\{x_1, x_2,\ldots, x_m\}$, $c(\bq)=\{x_1, x_2,\ldots, x_k\}$ and $t(\bq)=x_k$. By the identities (\ref{28101})
and \eqref{28102} we derive
\begin{equation}\label{id25032601}
\bq \approx x_1^2x_2^2\cdots x_k^2
\end{equation}
or
\begin{equation}\label{id25032602}
\bq \approx x_1^2x_2^2\cdots x_{k-1}^2x_k.
\end{equation}
Now we have
\begin{align*}
\bu
&\approx \bu+\bu_{i_1}+\bu_{i_2} \\
&\approx \bu+\bu_{i_1}+p(\bu_{i_2})t(\bu_{i_2})\\
&\approx \bu+\bu_{i_1}+p(\bu_{i_2})t(\bu_{i_2})+p(\bu_{i_2})\bu_{i_1} &&(\text{by}~\eqref{28105})\\
&\approx \bu+\bu_{i_1}+p(\bu_{i_2})t(\bu_{i_2})+(p(\bu_{i_2})\bu_{i_1})^2. &&(\text{by}~(\ref{28101}), (\ref{28102}))
\end{align*}
This implies the identity
\[
\bu \approx \bu+(p(\bu_{i_2})\bu_{i_1})^2.
\]
Furthermore, we can deduce
\begin{align*}
\bu
&\approx \bu+\bu_1+\cdots+\bu_n+(p(\bu_{i_2})\bu_{i_1})^2 \\
&\approx \bu+\bu_1+\cdots+\bu_n+x_1^2x_2^2\cdots x_m^2 &&(\text{by}~\eqref{28101}, \eqref{28104}, \eqref{28102})\\
&\approx \bu+\bu_1+\cdots+\bu_n+x_1^2+x_1^2x_2^2\cdots x_m^2 &&(\text{by}~\eqref{28101}, \eqref{28103})\\
&\approx \bu+\bu_1+\cdots+\bu_n+x_1^2+x_1^2x_2^2\cdots x_m^2+\bq. &&(\text{by}~\eqref{28106}, \eqref{28102}, \eqref{id25032601}, \eqref{id25032602})
\end{align*}
This derives the identity $\bu \approx \bu+\bq$.
\end{proof}

\begin{remark}
It is easy to verify that $S_{(4, 281)}$ is isomorphic to a subdirect product of $S_{4}$ and $S_{21}$.
So $\mathsf{V}(S_{(4, 281)}) = \mathsf{V}(S_{4}, S_{21})$.
On the other hand, it is easy
to see that both $M_2$ and $N_2$ can be embedded into $S_{21}$.
Also, $S_{21}$ satisfies the equational basis of $\mathsf{V}(M_2, N_2)$ that can be found in \cite{sr}.
Thus $\mathsf{V}(S_{21})=\mathsf{V}(M_2, N_2)$ and so $\mathsf{V}(S_{(4, 281)})=\mathsf{V}(S_{4}, M_2, N_2)$.
Since $M_2$ can be embedded into $S_{4}$, we therefore have
\[
\mathsf{V}(S_{(4, 281)})=\mathsf{V}(S_{4}, N_2).
\]
\end{remark}
\begin{cor}
The ai-semiring $S_{(4, 279)}$ is finitely based.
\end{cor}
\begin{proof}
It is easy to see that $S_{(4, 279)}$ and $S_{(4, 281)}$ have dual multiplications.
By Proposition $\ref{pro28101}$ we immediately deduce that $S_{(4, 279)}$ is finitely based.
\end{proof}

\begin{pro}
$S_{(4, 290)}$ and $S_{(4, 292)}$ are both finitely based.
\end{pro}
\begin{proof}
It is a routine matter to verify that $S_{(4, 292)}$ is isomorphic to a subdirect product of $S_{4}$ and $S_{22}$.
So $\mathsf{V}(S_{(4, 292)}) = \mathsf{V}(S_{4}, S_{22})$.
On the other hand, it is easy
to see that both $M_2$ and $D_2$ can be embedded into $S_{22}$.
Also, $S_{22}$ satisfies the equational basis of $\mathsf{V}(M_2, D_2)$ that can be found in \cite{sr}.
It follows that $\mathsf{V}(S_{22})=\mathsf{V}(M_2, D_2)$ and so $\mathsf{V}(S_{(4, 292)})=\mathsf{V}(S_{4}, M_2, D_2)$.
Since $M_2$ can be embedded into $S_{4}$, we therefore have
\[
\mathsf{V}(S_{(4, 292)})=\mathsf{V}(S_{4}, D_2).
\]
Finally, it follows that $\mathsf{V}(S_{(4, 48)})=\mathsf{V}(S_{4}, D_2)$ can be found in \cite{rlzc} and so
$\mathsf{V}(S_{(4, 292)})=\mathsf{V}(S_{(4, 48)})$. Thus $S_{(4, 292)}$ is finitely based.
Notice that $S_{(4, 290)}$ and $S_{(4, 292)}$ have dual multiplications.
So $S_{(4, 290)}$ is also finitely based.
\end{proof}

\begin{pro}
$S_{(4, 303)}$ and $S_{(4, 324)}$ are both finitely based.
\end{pro}
\begin{proof}
It is easily verified that $S_{(4, 303)}$ is isomorphic to a subdirect product of $S_{4}$ and $S_{29}$.
So $\mathsf{V}(S_{(4, 303)}) = \mathsf{V}(S_{4}, S_{29})$.
On the other hand, it is easy
to see that both $M_2$ and $R_2$ can be embedded into $S_{29}$.
Also, $S_{29}$ satisfies the equational basis of $\mathsf{V}(M_2, R_2)$ that can be found in \cite{sr}.
Hence $\mathsf{V}(S_{29})=\mathsf{V}(M_2, R_2)$ and so $\mathsf{V}(S_{(4, 303)})=\mathsf{V}(S_{4}, R_2, M_2)$.
Since $M_2$ can be embedded into $S_{4}$, we therefore have
\[
\mathsf{V}(S_{(4, 303)})=\mathsf{V}(S_{4}, R_2).
\]
From \cite[Remark 5.4]{rlzc} we know that $\mathsf{V}(S_{(4, 30)})=\mathsf{V}(S_{4}, R_2)$
and so
$\mathsf{V}(S_{(4, 303)})=\mathsf{V}(S_{(4, 30)})$.
By \cite[Proposition 5.3]{rlzc} it follows that $S_{(4, 303)}$ is finitely based.
Notice that $S_{(4, 324)}$ and $S_{(4, 303)}$ have dual multiplications.
We immediately deduce that $S_{(4, 324)}$ is finitely based.
\end{proof}

\begin{pro}
$S_{(4, 301)}$ and $S_{(4, 325)}$ are both finitely based.
\end{pro}
\begin{proof}
It is easy to check that $S_{(4, 325)}$ is isomorphic to a subdirect product of $S_{4}$ and $S_{23}$.
So $\mathsf{V}(S_{(4, 325)}) = \mathsf{V}(S_{4}, S_{23})$.
On the other hand, it is easy
to see that both $M_2$ and $L_2$ can be embedded into $S_{23}$.
Also, $S_{23}$ satisfies the equational basis of $\mathsf{V}(M_2, L_2)$ that can be found in \cite{sr}.
This implies that $\mathsf{V}(S_{23})=\mathsf{V}(M_2, L_2)$ and so $\mathsf{V}(S_{(4, 325)})=\mathsf{V}(S_{4}, M_2, L_2)$.
Notice that $M_2$ can be embedded into $S_{4}$. We therefore have
\[
\mathsf{V}(S_{(4, 325)})=\mathsf{V}(S_{4}, L_2).
\]
By \cite[Remark 5.7]{rlzc} we have that $\mathsf{V}(S_{(4, 47)})=\mathsf{V}(S_{4}, L_2)$
and so $\mathsf{V}(S_{(4, 325)})=\mathsf{V}(S_{(4, 47)})$.
Now it follows from  \cite[Proposition 5.6]{rlzc} that $S_{(4, 325)}$ is finitely based.
Since $S_{(4, 301)}$ and $S_{(4, 325)}$ have dual multiplications,
we immediately deduce that $S_{(4, 301)}$ is also finitely based.
\end{proof}

Let $\bp$ be a word. Then $S_2(\bp)$ denotes the set of all subwords of length $2$ of $\bp$.
If $\bv$ is an ai-semiring term such that $\bv=\bv_1+\bv_2+\cdots+\bv_k$ and $\bv_i\in X^+$, $1\leq i \leq k$,
we shall use $S_2(\bv)$ to denote the set $\bigcup_{1\leq i \leq k}S_2(\bv_i)$.
The following result can be found in \cite[Lemma 3.7]{yrzs}.
\begin{lem}\label{lem5301}
Let $\bu\approx \bu+\bq$ be a nontrivial ai-semiring identity such that
$\bu=\bu_1+\bu_2+\cdots+\bu_n$ and $\bu_i, \bq \in X_c^+$, $1\leq i \leq n$.
Suppose that $\bu\approx \bu+\bq$ holds in $S_{53}$.
Then $L_{\geq 2}(\bu)\neq \emptyset$, $c(\bq)\subseteq c(\bu)$,
and for any $\bw\in S_2(\bq)$ there exists $\bw'\in S_2(\bu)$ such that $c(\bw')\subseteq c(\bw)$.
\end{lem}

\begin{pro}\label{pro35701}
$\mathsf{V}(S_{(4, 357)})$ is the ai-semiring variety defined by the identities
\begin{align}
xyz &\approx yxz; \label{35702}\\
xy &\approx xy+y; \label{35703}\\
x^2y+z & \approx x^2y+z+xz;\label{35707}\\
x_1x_2+x_2x_3+x_4x_5 & \approx x_1x_2+x_2x_3+x_4x_5+x_4;\label{35710}\\
x_1x_2+x_2x_3+x_1x_3 & \approx x_1x_2+x_2x_3+x_1x_3+x_1x_2x_3;\label{35711}\\
x_1x_3+x_2x_3+x_1x_2x_4 & \approx x_1x_3+x_2x_3+x_1x_2x_4+x_1x_2x_3;\label{35712}\\
x_1x_2+x_2x_3+x_4x_5x_6x_7 & \approx x_1x_2+x_2x_3+x_4x_5x_6x_7+x_5x_6;\label{35714}\\
x_1x_2+x_2x_3+x_4x_5^2x_6+x_7x_8x_9 & \approx x_1x_2+x_2x_3+x_4x_5^2x_6+x_7x_8x_9+x_5x_8,\label{35715}
\end{align}
where $x_1$ may be empty in $(\ref{35710})$,
$x_1$, $x_4$ and $x_7$ may be empty in $(\ref{35714})$, $x_1$, $x_4$, $x_6$, $x_7$ and $x_9$ may be empty in $(\ref{35715})$.
\end{pro}
\begin{proof}
It is a routine matter to verify that $S_{(4, 357)}$ satisfies the identities \eqref{35702}--\eqref{35715}.
In the remainder we shall show that every ai-semiring identity of $S_{(4, 357)}$
is derivable from \eqref{35702}--\eqref{35715} and the identities defining $\mathbf{AI}$.
Let $\bu \approx \bu+\bq$ be such a nontrivial identity, where
$\bu=\bu_1+\bu_2+\cdots+\bu_n$ and $\bu_i, \bq \in X^+$, $1 \leq i \leq n$.
It is easy to check that $S_{(4, 357)}$ is isomorphic to a subdirect product of $S_{4}$ and $S_{53}$.
So both $S_{4}$ and $S_{53}$ satisfy $\bu \approx \bu+\bq$.
By Lemma \ref{lem401} we have that $c(\bq)\subseteq  c(\bu)$ and $L_{\geq 2}(\bu)\neq \emptyset$.

\textbf{Case 1.} $c(p(\bu))\cap t(\bu)=\emptyset$.
Then by Lemma \ref{lem401} $c(p(\bu+\bq))\cap t(\bu+\bq)=\emptyset$, and so $m(t(\bq), \bq)=1$.
If $\ell(\bq)=1$, then there exists $\bu_j\in L_{\geq 2}(\bu)$ such that $\bu_j=p(\bu_j)\bq$, and so
\[
\bu \approx \bu+\bu_j \approx \bu+p(\bu_j) \bq \stackrel{\eqref{35703}}\approx \bu+p(\bu_j)\bq+\bq\approx \bu+\bq.
\]
Now consider the case that $\ell(\bq)\geq 2$.
Then $\bq=x_1x_2\cdots x_n$ for some $x_1, x_2, \ldots, x_n \in X$ and $n\geq 2$,
and so $x_n\neq x_i$ for all $1\leq i <n$.

Let $1\leq i <n$. It follows from Lemma \ref{lem5301} that
there exists $\bw_i\in S_2(\bu)$ such that $c(\bw_i)\subseteq \{x_i, x_n\}$ and so $\bw_i=x_i^2$ or $\bw_i=x_ix_n$.
If $\bw_i=x_i^2$, then there exists $\bu_k \in \bu$ such that
$\bu_k=x_i^2\bu_k'$ for some $\bu_k' \in X^+$.
Since $c(\bq)\subseteq c(\bu)$, we have $\bu_\ell=p(\bu_\ell)x_n$ for some $\bu_\ell \in \bu$, and so

\begin{align*}
\bu
&\approx \bu+\bu_k+\bu_\ell \\
&\approx \bu+x_i^2\bu_k'+p(\bu_\ell)x_n\\
&\approx \bu+x_i^2\bu_k'+p(\bu_\ell)x_n+x_ip(\bu_\ell)x_n &&(\text{by}~\eqref{35707})\\
&\approx \bu+x_i^2\bu_k'+p(\bu_\ell)x_n+p(\bu_\ell)x_ix_n &&(\text{by}~\eqref{35702})\\
&\approx \bu+x_i^2\bu_k'+p(\bu_\ell)x_n+p(\bu_\ell)x_ix_n+x_ix_n. &&(\text{by}~\eqref{35703})
\end{align*}
This implies the identity $\bu \approx \bu+x_ix_n$.
If $\bw_i=x_ix_n$, then there exists $\bu_k \in \bu$ such that
$\bu_k=\bu_k'x_ix_n$ for some $\bu_k'\in X^*$.
Now we have
\[
\bu \approx \bu+\bu_k \approx \bu+\bu_k'x_ix_n \stackrel{(\ref{35703})}\approx \bu+\bu_k'x_ix_n+x_ix_n.
\]
So we can derive the identity $\bu \approx \bu+x_ix_n$.

Let $i$ and $j$ be distinct numbers such that $1\leq i, j<n$.
By Lemma \ref{lem5301} there exists $\bu_t\in \bu$ such that
$c(\bw)\subseteq \{x_i, x_j\}$ for some $\bw\in S_2(p(\bu_t))$.
It is easy to see that $\bw=x_i^2$ or $\bw=x_j^2$ or $\bw=x_ix_j$.
If $\bw=x_ix_j$, then by \eqref{35702} we may write $\bu_t=x_ix_j\bu_t'$ for some $\bu_t'\in X^+$.
If $\bw=x_i^2$, then $\bu_t=x_i^2\bu_t'$ for some $\bu_t'\in X^+$.
Also, there exists $\bu_r\in \bu$ such that $\bu_r=x_j\bu_r'$ for some $\bu_r'\in X^+$.
Now we have
\[
\bu \approx \bu+\bu_t+\bu_r\approx \bu+x_i^2\bu_t'+x_j\bu_r'\stackrel{\eqref{35707}}\approx \bu+x_i^2\bu_t'+x_j\bu_r'+x_ix_j\bu_r',
\]
and so $\bu \approx \bu+x_ix_j\bu_r'$ is derived.
We can show in a closely similar way that the identity $\bu \approx \bu+x_ix_j\bu_s'$ for some $\bu_s'\in X^+$ is derived if $\bw=x_j^2$.
Therefore, we can always derive the identity $\bu \approx \bu+x_ix_j\bp_{ij}$ for some $\bp_{ij}\in X^+$.

We shall show by induction on $k$ that the identities \eqref{35702}--\eqref{35715} can
derive
\[
\bu \approx \bu+x_1\cdots x_k(x_{k+1}+\cdots+x_{n-1})x_n
\]
for all $1\leq k<n-1$. Indeed, if $k=1$, then for any $1<j<n$, we have
\[
\bu \approx \bu+x_1x_n+x_jx_n+x_1x_j\bp_{1j}\stackrel{\eqref{35712}}\approx \bu+x_1x_n+x_jx_n+x_1x_j\bp_{1j}+x_1x_jx_n,
\]
and so $\bu \approx \bu+x_1x_jx_n$ is derived. This implies the identity
\[
\bu \approx \bu+x_1x_2x_n+\cdots+x_1x_{n-1}x_n\approx\bu+x_1(x_2+\cdots+x_{n-1})x_n.
\]
Let $2\leq k<n-1$. Suppose that the result is true for $k-1$, that is, the identities \eqref{35702}--\eqref{35715} can
derive
\[
\bu \approx \bu+x_1\cdots x_{k-1}(x_{k}+\cdots+x_{n-1})x_n.
\]
Then for any $k< j<n$, one can obtain the identity
\[
\bu \approx \bu+x_1\cdots x_{k-1}x_{k}x_n+x_1\cdots x_{k-1}x_{j}x_n.
\]
Furthermore, we have
\begin{align*}
\bu
&\approx \bu+x_1\cdots x_{k-1}x_{k}x_n+x_1\cdots x_{k-1}x_{j}x_n \\
&\approx \bu+x_1\cdots x_{k-1}x_{k}x_n+x_1\cdots x_{k-1}x_{j}x_n+x_{k}x_{j}\bp_{kj}\\
&\approx \bu+x_{k}x_1\cdots x_{k-1}x_n+x_{j}x_1\cdots x_{k-1}x_n+x_{k}x_{j}\bp_{kj}&&(\text{by}~\eqref{35702})\\
&\approx \bu+x_{k}x_1\cdots x_{k-1}x_n+x_{j}x_1\cdots x_{k-1}x_n+x_{k}x_{j}\bp_{kj}+x_{k}x_jx_1\cdots x_{k-1}x_n &&(\text{by}~\eqref{35712})\\
&\approx \bu+x_{k}x_1\cdots x_{k-1}x_n+x_{j}x_1\cdots x_{k-1}x_n+x_{k}x_{j}\bp_{kj}+x_1\cdots x_{k-1}x_{k}x_jx_n. &&(\text{by}~\eqref{35702})
\end{align*}
This implies the identity
\[
\bu\approx \bu+x_1\cdots x_{k}x_jx_n,
\]
and so
\[
\bu \approx \bu+x_1\cdots x_k(x_{k+1}+\cdots+x_{n-1})x_n
\]
is proved.
Take $k=n-2$. We obtain the identity
\[
\bu\approx \bu+x_1x_2\cdots x_n\approx \bu+\bq.
\]

\textbf{Case 2.} $c(p(\bu))\cap t(\bu)\neq\emptyset$.
Then $t(\bu_i) \in c(p(\bu_j))$ for some $\bu_i, \bu_j\in\bu$,
and so by \eqref{35702} we may write $\bu_j=t(\bu_i)\bu_j'$ for some $\bu_j'\in X^+$.
If $\ell(\bq)=1$, then there exists $\bu_k\in L_{\geq2}(\bu)$ such that
$\bq\in c(\bu_k)$ and so
$\bu_k=\bu_k'\bq$ or $\bu_k=\bq\bu_k'$ for some $\bu_k'\in X^+$.
If $\bu_k=\bu_k'\bq$, then
\[
\bu \approx \bu+\bu_k \approx \bu+\bu_k'\bq \stackrel{\eqref{35703}}\approx \bu+\bu_k'\bq+\bq\approx \bu+\bq.
\]
If $\bu_k=\bq\bu_k'$, then
\[
\bu \approx \bu+\bu_i+\bu_j+\bu_k  \stackrel{\eqref{35710}}\approx \bu+p(\bu_i)t(\bu_i)+t(\bu_i)\bu_j'+\bq\bu_k'+\bq\approx \bu+\bq.
\]
Now we consider the case that $\ell(\bq)\geq 2$.
Let us write $\bq=x_1x_2\cdots x_n$, where $n\geq2$.
Let $i$ and $j$ be distinct numbers such that $1\leq i, j \leq n$.
By Lemma \ref{lem5301} there exists $\bu_k\in \bu$ such that
$c(\bw)\subseteq \{x_i, x_j\}$ for some $\bw \in S_2(\bu_k)$,
and so $\bw$ is in $\{x_i^2, x_j^2, x_ix_j, x_jx_i\}$.

If $\bw=x_i^2$, then $\bu_k=\bp_1x_i^2\bp_2$ for some $\bp_1, \bp_2\in X^*$.
Since $c(\bq)\subseteq c(\bu)$, we have that $\bu_\ell=\bp_3x_j\bp_4$ for some $\bu_\ell \in \bu$ and $\bp_3, \bp_4\in X^*$. Then
\begin{align*}
\bu
& \approx \bu+\bu_i+\bu_j+\bu_k+\bu_\ell \\
& \approx \bu+p(\bu_i)t(\bu_i)+t(\bu_i)\bu_j'+\bp_1x_i^2\bp_2+\bp_3x_j\bp_4 \\
& \approx \bu+p(\bu_i)t(\bu_i)+t(\bu_i)\bu_j'+\bp_1x_i^2\bp_2+\bp_3x_j\bp_4 +x_ix_j, &&(\text{by}~(\ref{35715}))
\end{align*}
and so $\bu\approx \bu+x_ix_j$ can be derived.
If $\bw=x_ix_j$, then $\bu_k=\bp_1x_ix_j\bp_2$ for some $\bp_1, \bp_2\in X^*$, and so
\begin{align*}
\bu
& \approx \bu+\bu_i+\bu_j+\bu_k\\
& \approx \bu+p(\bu_i)t(\bu_i)+t(\bu_i)\bu_j'+\bp_1x_ix_j\bp_2\\
& \approx \bu+p(\bu_i)t(\bu_i)+t(\bu_i)\bu_j'+\bp_1x_ix_j\bp_2+x_ix_j, &&(\text{by}~(\ref{35714}))
\end{align*}
and so $\bu\approx \bu+x_ix_j$ is proved.
One can use a similar argument to show that it is true if $\bw=x_j^2$ or $x_jx_i$.
So we can always derive $\bu\approx \bu+x_ix_j$.
Now we repeat the process in Case 1 and show that
the identities \eqref{35702}--\eqref{35715} can
derive
\[
\bu \approx \bu+x_1\cdots x_k(x_{k+1}+\cdots+x_{n-1})x_n
\]
for all $1\leq k<n-1$.
So the identity $\bu \approx \bu+\bq$ is derived.
\end{proof}

\begin{cor}
The ai-semiring $S_{(4, 358)}$ is finitely based.
\end{cor}
\begin{proof}
It is easy to see that $S_{(4, 358)}$ and $S_{(4, 357)}$ have dual multiplications.
By Proposition $\ref{pro35701}$ we immediately deduce that $S_{(4, 358)}$ is finitely based.
\end{proof}

The following result can be found in \cite[Lemma 3.1]{yrzs}.
\begin{lem}\label{lem5701}
Let $\bu\approx \bu+\bq$ be a nontrivial ai-semiring identity such that
$\bu=\bu_1+\bu_2+\cdots+\bu_n$ and $\bu_i, \bq \in X^+$, $1\leq i \leq n$.
If $\bu\approx \bu+\bq$ holds in $S_{57}$,
then $L_{\geq 2}(\bu)\neq \emptyset$,
$c(p(\bq))\subseteq c(p(\bu))$ and $t(\bq)\in c(\bu)$.
\end{lem}
\begin{pro}\label{pro36201}
$\mathsf{V}(S_{(4, 362)})$ is the ai-semiring variety defined by the identities
\begin{align}
x^2y& \approx xy; \label{36201}\\
x^2y^2&\approx x^2+y^2;\label{36203}\\
x+yz& \approx yx+yz. \label{36204}
\end{align}
\end{pro}
\begin{proof}
It is easy to verify that $S_{(4, 362)}$ satisfies the identities \eqref{36201}--\eqref{36204}.
In the proof of Proposition \ref{pro28101},
the identities \eqref{36201} and \eqref{36203} are used to derive the identity
\begin{align}
xyz \approx yxz. \label{36202}
\end{align}
To complete the proof it remains to show that every ai-semiring identity of $S_{(4, 362)}$
can be derived by \eqref{36201}--\eqref{36202} and the identities defining $\mathbf{AI}$.
Let $\bu \approx \bu+\bq$ be such a nontrivial identity,
where $\bu=\bu_1+\bu_2+\cdots+\bu_n$ and $\bu_i, \bq \in X^+$, $1 \leq i \leq n$.
It is easy to check that $S_{(4, 362)}$ is isomorphic to
a subdirect product of $S_{4}$ and $S_{57}$.
So both $S_{4}$ and $S_{57}$ satisfy $\bu \approx \bu+\bq$.
By Lemmas \ref{lem401} and \ref{lem5701}
we have that $c(p(\bq))\subseteq c(p(\bu))$, $t(\bq)\in c(\bu)$, and $L_{\geq 2}(\bu)$ is nonempty.
Suppose that
$
c(p(\bu))=\{x_1, x_2, \ldots, x_m\}
$
and that
$
t(\bu)=\{y_1, y_2, \ldots, y_n\}.
$
Then
\[
c(p(\bq)) \subseteq \{x_1, x_2, \ldots, x_m\}, t(\bq)\in \{x_1, x_2, \ldots, x_m\}\cup \{y_1, y_2, \ldots, y_n\}.
\]
By the identities \eqref{36201}, \eqref{36202} and \eqref{36204} we derive
\begin{equation}\label{id25032620}
\bu \approx \bu+x_1^2x_2^2\cdots x_m^2(y_1+y_2+\cdots+y_n).
\end{equation}

\textbf{Case 1.} $c(p(\bu))\cap t(\bu)=\emptyset$. Then by Lemma \ref{lem401} $c(p(\bu+\bq))\cap t(\bu+\bq)=\emptyset$
and so
\[
\{x_1, x_2, \ldots, x_m\}\cap \{y_1, y_2, \ldots, y_n\}=\emptyset.
\]
This implies that $m(t(\bq), \bq)=1$, $t(\bq)\notin c(p(\bq))$ and $t(\bq)\in t(\bu)$.
If $\ell(\bq)=1$, then there exists $\bu_t\in L_{\geq2}(\bu)$ such that $t(\bq)=t(\bu_t)$ and so $\bu_t=p(\bu_t)t(\bq)$.
Furthermore, we have
\[
\bu \approx \bu+\bu_t \approx \bu+p(\bu_t)t(\bq) \stackrel{(\ref{36204})}\approx \bu+p(\bu_t)t(\bq)+t(\bq)\approx  \bu+p(\bu_t)t(\bq)+\bq.
\]
If $\ell(\bq)\geq2$, we may assume that $c(p(\bq))=\{x_1, x_2,\ldots, x_s\}$ and $t(\bq)=y_1$.
By the identities (\ref{36201}) and (\ref{36202}) we deduce
\begin{equation}\label{id25032621}
\bq \approx x_1^2x_2^2\cdots x_s^2y_1.
\end{equation}
Now we have
\begin{align*}
\bu
&\approx \bu+x_1^2x_2^2 \cdots x_m^2y_1&&(\text{by}~\eqref{id25032620})\\
&\approx \bu+(x_1^2x_2^2\cdots x_s^2+x_{s+1}^2\cdots x_m^2)y_1&&(\text{by}~\eqref{36203})\\
&\approx \bu+x_1^2x_2^2\cdots x_s^2y_1+x_{s+1}^2\cdots x_m^2y_1\\
&\approx \bu+\bq+x_{s+1}^2\cdots x_m^2y_1.  &&(\text{by}~\eqref{id25032621})
\end{align*}
This implies the identity $\bu \approx \bu+\bq$.

\textbf{Case 2.} $c(p(\bu))\cap t(\bu)\neq\emptyset$.
Then $t(\bu_{i_1})\in c(p(\bu_{i_2}))$ for some $\bu_{i_1}, \bu_{i_2} \in \bu$
and so
\[
\{x_1, x_2, \ldots, x_m\}\cap \{y_1, y_2, \ldots, y_n\}\neq\emptyset.
\]
Furthermore, we have
\begin{align*}
\bu
&\approx \bu+x_1^2x_2^2\cdots x_m^2(y_1+\cdots+y_n)&&(\text{by}~\eqref{id25032620})\\
&\approx \bu+x_1^2x_2^2\cdots x_m^2(y_s+\cdots+y_t)+x_1^2x_2^2\cdots x_m^2. &&(\text{by}~\eqref{36201}, \eqref{36202}, \eqref{36203})
\end{align*}
This derives the identity
\begin{equation}\label{id25032630}
\bu \approx \bu+x_1^2x_2^2\cdots x_m^2.
\end{equation}
If $\ell(\bq)=1$ and $t(\bq)\in c(p(\bu))$, then $t(\bq)=\bq$. Take $\bq=x_1$. We have
\begin{align*}
\bu
&\approx \bu+x_1^2x_2^2\cdots x_m^2&&(\text{by}~\eqref{id25032630})\\
&\approx \bu+x_2^2\cdots x_m^2x_1^2&&(\text{by}~\eqref{36203})\\
&\approx \bu+x_2^2\cdots x_m^2x_1^2+x_1&&(\text{by}~\eqref{36204})\\
&\approx \bu+x_2^2\cdots x_m^2x_1^2+\bq.
\end{align*}
If $\ell(\bq)=1$ and $t(\bq)\in t(\bu)$, then $t(\bq)=\bq=t(\bu_k)$ for some $\bu_k\in L_{\geq2}(\bu)$. Now we have
\[\bu
\approx \bu+\bu_k \approx \bu+p(\bu_k)t(\bq) \stackrel{\eqref{36204}} \approx \bu+p(\bu_k)t(\bq)+t(\bq)\approx \bu+p(\bu_k)t(\bq)+\bq.
\]
If $\ell(\bq)\geq2$ and $t(\bq)\in c(p(\bu))$, then
\begin{align*}
\bu
&\approx \bu+x_1^2x_2^2\cdots x_m^2&&(\text{by}~\eqref{id25032630})\\
&\approx \bu+x_1^2x_2^2\cdots x_m^2+\bq. &&(\text{by}~\eqref{36203}, \eqref{36204}, \eqref{36202})
\end{align*}
If $\ell(\bq)\geq2$ and $t(\bq)\in t(\bu)$,
we may assume that $c(p(\bq))=\{x_1, x_2,\ldots, x_s\}$ and $t(\bq)=y_{_k}$.
By the identities \eqref{36201} and \eqref{36202} we deduce
\begin{equation}\label{id25032650}
\bq \approx x_1^2x_2^2\cdots x_s^2y_{_k}.
\end{equation}
Now we have
\begin{align*}
\bu
&\approx \bu+x_1^2x_2^2\cdots x_m^2y_{_k} &&(\text{by}~\eqref{id25032620})\\
&\approx \bu+(x_1^2x_2^2\cdots x_s^2+x_{s+1}^2\cdots x_m^2)y_{_k}&&(\text{by}~\eqref{36204})\\
&\approx \bu+ x_1^2x_2^2\cdots x_s^2y_{_k}+x_{s+1}^2\cdots x_m^2y_{_k}\\
&\approx \bu+\bq+x_{s+1}^2\cdots x_m^2y_{_k}.   &&(\text{by}~\eqref{id25032650})
\end{align*}
This derives the identity $\bu \approx \bu+\bq$.
\end{proof}
\begin{cor}
The ai-semiring $S_{(4, 365)}$ is finitely based.
\end{cor}
\begin{proof}
It is easy to see that $S_{(4, 365)}$ and $S_{(4, 362)}$ have dual multiplications.
By Proposition $\ref{pro36201}$ we immediately deduce that $S_{(4, 365)}$ is finitely based.
\end{proof}

\section{Equational bases of $S_{(4, 285)}$, $S_{(4, 379)}$, $S_{(4, 380)}$ and $S_{(4, 385)}$}
In this section we focus on the finite basis problem for
some 4-element ai-semirings that relate to one of $S_{10}$, $S_{44}$, $S_{46}$ and $S_{47}$.
The following result, which is a consequence of \cite[Corollary 2.13]{rz16} and its proof,
provides a solution of the equational problem for $S_{10}$.
\begin{lem}\label{lem1001}
Let $\bu\approx \bu+\bq$ be a nontrivial ai-semiring identity such that
$\bu=\bu_1+\bu_2+\cdots+\bu_n$ and $\bu_i, \bq \in X^+$, $1\leq i \leq n$.
Then $\bu\approx \bu+\bq$ is satisfied by $S_{10}$ if and only if $c(\bq)\subseteq  c(\bu)$ and
$r(\bq)=r(\bu_{i_1}\bu_{i_2} \cdots\bu_{i_{3^\ell}})$ for some $\bu_{i_1}, \bu_{i_2}, \ldots, \bu_{i_{3^\ell}}\in \bu$,
where $r(\bq)=\{x\mid x\in c(\bq), m(x,\bq)~\textrm{is an odd number}\}$.
\end{lem}
\begin{pro}\label{pro28501}
$\mathsf{V}(S_{(4, 285)})$ is the ai-semiring variety defined by the identities
\begin{align}
x^3y& \approx xy; \label{28501}\\
xy& \approx yx; \label{28502}\\
xy^2& \approx xy^2+x^3;    \label{28503}\\
x+y+z& \approx x+y+z+xyz.    \label{28504}
\end{align}
\end{pro}
\begin{proof}
It is easy to check that $S_{(4, 285)}$ satisfies the identities \eqref{28501}--\eqref{28504}.
In the remainder we shall prove that every ai-semiring identity of $S_{(4, 285)}$
can be derived by \eqref{28501}--\eqref{28504} and the identities defining $\mathbf{AI}$.
Let $\bu \approx \bu+\bq$ be such a nontrivial identity,
where $\bu=\bu_1+\bu_2+\cdots+\bu_n$ and $\bu_i, \bq \in X^+$, $1 \leq i \leq n$.
It is easy to verify that $S_{(4, 285)}$ is isomorphic to a subdirect product of $N_2$ and $S_{10}$.
So both $N_2$ and $S_{10}$ satisfy $\bu \approx \bu+\bq$.
This implies that $\ell(\bq)\geq 2$.
By Lemma \ref{lem1001} we have that $c(\bq)\subseteq c(\bu)$ and $r(\bq)=r(\bu_{i_1}\bu_{i_2} \cdots\bu_{i_{3^\ell}})$
for some $\bu_{i_1}, \bu_{i_2}, \cdots, \bu_{i_{3^\ell}}\in \bu$.
Assume that $c(\bq)=\{x_1, x_2, \ldots, x_m, y_1, y_2, \ldots, y_n\}$,
where $m(x_i, \bq)$ is an odd number and $m(y_j, \bq)$ is an even number
for all $1\leq i \leq m$ and $1\leq j \leq n$.
By the identities \eqref{28501} and \eqref{28502} we deduce
\begin{equation}\label{id25032660}
\bq \approx x_1x_2 \cdots x_my_1^2y_2^2 \cdots y_n^2.
\end{equation}
Now we have
\begin{align*}
\bu
&\approx \bu+\bu_1^{3^\ell+1}\bu_{i_1}^3\bu_{i_2}^3\cdots \bu_{i_{3^\ell}}^3\bu_1^2\bu_2^2\cdots \bu_n^2&&(\text{by}~\eqref{28502}, \eqref{28504})\\
&\approx \bu+x_1x_2\cdots x_my_1^2y_2^2\cdots y_n^2\bp_1^2&&(\text{by}~\eqref{28502})\\
&\approx \bu+x_1x_2\cdots x_my_1^2y_2^2\cdots y_n^2\bp_1^2+(x_1x_2\cdots x_my_1^2y_2^2\cdots y_n^2)^3&&(\text{by}~\eqref{28503})\\
&\approx \bu+x_1x_2\cdots x_my_1^2y_2^2\cdots y_n^2\bp_1^2+\bq^3 &&(\text{by}~\eqref{id25032660})\\
&\approx \bu+x_1x_2\cdots x_my_1^2y_2^2\cdots y_n^2\bp_1^2+\bq.  &&(\text{by}~\eqref{28501}, \eqref{28502})
\end{align*}
This implies the identity $\bu \approx \bu+\bq$.
The proof is analogous if $m(x, \bq)$ is an odd number for all $x\in c(\bq)$
or $m(x, \bq)$ is an even number for all $x\in c(\bq)$.
\end{proof}

\begin{lem}\label{lem4401}
Let $\bu\approx \bu+\bq$ be a nontrivial ai-semiring identity such that
$\bu=\bu_1+\bu_2+\cdots+\bu_n$ and $\bu_i, \bq \in X_c^+$, $1\leq i \leq n$.
Suppose that $\bu\approx \bu+\bq$ is satisfied by $S_{44}$.
Then $\ell(\bq)\geq 2$ and $D_\bq(\bu)\neq\emptyset$.
If $M_{1}(\bq)\neq\emptyset$,
then for any $x \in M_{1}(\bq)$, there exists $\bu_i \in D_\bq(\bu)$ such that $m(x, \bu_i)\leq 1$.
\end{lem}
\begin{proof}
This can be found in \cite[Lemma 7.12]{yrzs}.
\end{proof}

\begin{pro}\label{pro37901}
$\mathsf{V}(S_{(4, 379)})$ is the ai-semiring variety defined by the identities
\begin{align}
xy& \approx yx; \label{37901}\\
xy& \approx x^2y+xy^2;    \label{37905}\\
xy& \approx xy+xyz;    \label{37903}\\
x+yz& \approx x+yz+xz.    \label{37904}
\end{align}
\end{pro}
\begin{proof}
It is easy to check that $S_{(4, 379)}$ satisfies the identities \eqref{37901}--\eqref{37904}.
By identifying $y$ and $x$ in \eqref{37905}, one can derive the identity
\begin{align}
x^3& \approx x^2. \label{37902}
\end{align}
In the remainder it is enough to prove that every ai-semiring identity of $S_{(4, 379)}$
is derivable from \eqref{37901}--\eqref{37902} and the identities defining $\mathbf{AI}$.
Let $\bu \approx \bu+\bq$ be such a nontrivial identity,
where $\bu=\bu_1+\bu_2+\cdots+\bu_n$ and $\bu_i, \bq\in X^+$, $1 \leq i \leq n$.
It is a routine matter to verify that $S_{(4, 379)}$ is isomorphic to
a subdirect product of $T_2$ and $S_{44}$.
So both $T_2$ and $S_{44}$ satisfy $\bu \approx \bu+\bq$.
This implies that $\ell(\bu_j)\geq 2$ for some $\bu_j\in \bu$.
By Lemma \ref{lem4401} we have that $\ell(\bq)\geq 2$ and $D_\bq(\bu)\neq\emptyset$.

\textbf{Case 1.}
$M_1(\bq)$ is empty. Then $m(x, \bq)\geq 2$ for all $x\in c(\bq)$.
Take $\bu_i$ in $D_{\bq}(\bu)$.
Then
\begin{align*}
\bu
&\approx \bu+\bu_i+\bu_j\\
&\approx \bu+\bu_i+\bu_j+\bu_j\bq &&(\text{by}~(\ref{37903}))\\
&\approx \bu+\bu_i+\bu_j+\bu_j\bq+\bu_i\bq &&(\text{by}~(\ref{37904}))\\
&\approx \bu+\bu_i+\bu_j+\bu_j\bq+\bq. &&(\text{by}~(\ref{37901}),(\ref{37902}))
\end{align*}
This derives $\bu \approx \bu+\bq$.

\textbf{Case 2.}
$M_1(\bq)$ is nonempty. It follows from Lemma \ref{lem4401} that
for any $x \in M_{1}(\bq)$, there exists $\bu_i \in D_\bq(\bu)$ such that $m(x, \bu_i)\leq 1$.
Assume that
\[
c(\bq)=\{x_1, \ldots, x_n, y_1, \ldots, y_m\},
\]
where $m(x_i, \bq )\geq 2$, $m(y_j, \bq )=1$, $0 \leq i \leq n$, $1 \leq j \leq m$.
Then
\begin{align*}
\bq
&\approx x_1^2x_2^2\cdots x_n^2y_1y_2\cdots y_m &&(\text{by}~ \eqref{37901}, \eqref{37902})\\
&\approx\sum\limits_{1\leq j\leq m}x_1^2x_2^2\cdots x_n^2y_1^2\cdots y_{j-1}^2y_{j+1}^2\cdots y_m^2y_j.
&&(\text{by}~\eqref{37901}, \eqref{37905}, \eqref{37902})
\end{align*}
So we only need to consider the case that $\bq=x_1^2x_2^2\cdots x_k^2y$.
Then by Lemma \ref{lem4401} there exists $\bu_{\ell} \in D_\bq(\bu)$ such that $m(y, \bu_{\ell})\leq 1$.
Now we have
\begin{align*}
\bu
&\approx \bu+\bu_{\ell}+\bu_j\\
&\approx \bu+\bu_{\ell}+\bu_j+\bu_j\bq_1 &&(\text{by}~(\ref{37903}))\\
&\approx \bu+\bu_{\ell}+\bu_j+\bu_j\bq_1+\bu_{\ell}\bq_1 &&(\text{by}~(\ref{37904}))\\
&\approx \bu+\bu_{\ell}+\bu_j+\bu_j\bq_1+\bq, &&(\text{by}~(\ref{37901}),(\ref{37902}))
\end{align*}
where $\bq_1=\bq$ if $m(y, \bu_{\ell})=0$, and $\bq_1=x_1^2x_2^2\cdots x_k^2$ if $m(y, \bu_{\ell})=1$.
This derives the identity $\bu \approx \bu+\bq$.
\end{proof}

\begin{lem}\label{lem4601}
Let $\bu\approx \bu+\bq$ be a nontrivial ai-semiring identity such that
$\bu=\bu_1+\bu_2+\cdots+\bu_n$ and $\bu_i, \bq \in X^+$, $1\leq i \leq n$.
Suppose that $\bu\approx \bu+\bq$ is satisfied by $S_{46}$.
Then $\ell(\bq)\geq 2$.
If $m(t(\bq), \bq)=1$,
then there exists $\bu_i \in D_\bq(\bu)$ such that
$t(\bq)\notin c(p(\bu_i))$.
\end{lem}
\begin{proof}
This can be found in \cite[Lemma 7.15]{yrzs}.
\end{proof}

\begin{pro}\label{pro38001}
$\mathsf{V}(S_{(4, 380)})$ is the ai-semiring variety defined by the identities
\begin{align}
xyz& \approx yxz; \label{38001}\\
x^2y& \approx xy; \label{38002}\\
x_1+x_2x_3& \approx x_1+x_2x_3+x_1x_4+x_5x_1.    \label{38005}
\end{align}
\end{pro}
\begin{proof}
It is easy to check that $S_{(4, 380)}$ satisfies the identities \eqref{38001}--\eqref{38005}.
In the remainder we need only prove that every ai-semiring identity of $S_{(4, 380)}$
can be derived by \eqref{38001}--\eqref{38005} and the identities defining $\mathbf{AI}$.
Let $\bu \approx \bu+\bq$ be such a nontrivial identity,
where $\bu=\bu_1+\bu_2+\cdots+\bu_n$ and $\bu_i, \bq\in X^+$, $1 \leq i \leq n$.
It is a routine matter to verify that $S_{(4, 380)}$ is isomorphic to
a subdirect product of $T_2$ and $S_{46}$.
So both $T_2$ and $S_{46}$ satisfy $\bu \approx \bu+\bq$.
This implies that $\ell(\bu_j)\geq 2$ for some $\bu_j\in \bu$.
By Lemma \ref{lem4601} we obtain that $\ell(\bq)\geq 2$.
Also, if $m(t(\bq), \bq)=1$,
then there exists $\bu_i \in D_\bq(\bu)$ such that
$t(\bq)\notin c(p(\bu_i))$.
Since $D_2$ is isomorphic to $\{2, 3\}$, it follows that $D_\bq(\bu)\neq\emptyset$.
Take $\bu_i$ in $D_\bq(\bu)$, and $t(\bq)\notin c(p(\bu_i))$ if $m(t(\bq), \bq)=1$.
Now we have
\[
\bu \approx \bu+\bu_i+\bu_j \stackrel{\eqref{38005}}\approx \bu+\bu_i+\bu_j+\bu_i\bq+p(\bq)\bu_i.
\]
If $m(t(\bq), \bq)\geq 2$ or $m(t(\bq), \bq)=1$, $m(t(\bq),\bu_i)=0$,
then the identities \eqref{38001} and \eqref{38002} imply
\[
\bu_i\bq \approx \bq.
\]
If $m(t(\bq), \bq)=1$, $m(t(\bq),\bu_i)=1$ and $t(\bq)=t(\bu_i)$, then
the identities \eqref{38001} and \eqref{38002} imply
\[
p(\bq)\bu_i \approx \bq.
\]
This derives the identity $\bu \approx \bu+\bq$.
\end{proof}

It is easy to see that $S_{(4, 383)}$ and $S_{(4, 380)}$ have dual multiplications.
By Proposition $\ref{pro38001}$ we immediately deduce
\begin{cor}
The ai-semiring $S_{(4, 383)}$ is finitely based.
\end{cor}

\begin{lem}\label{lem4701}
Let $\bu\approx \bu+\bq$ be a nontrivial ai-semiring identity such that
$\bu=\bu_1+\bu_2+\cdots+\bu_n$ and $\bu_i, \bq \in X^+$, $1\leq i \leq n$.
Suppose that $\bu\approx \bu+\bq$ is satisfied by $S_{47}$.
Then either $\ell(\bq)\geq 3$ or $\ell(\bq)=2$, $L_{\leq2}(\bu) \cap D_\bq(\bu)$ is nonempty.
\end{lem}
\begin{proof}
Suppose that $S_{47}$ satisfies the nontrivial identities $\bu\approx \bu+\bq$.
Since $N_2$ is isomorphic to $\{1, 2\}$, it follows that $N_2$ satisfies $\bu\approx \bu+\bq$ and so $\ell(\bq)\geq 2$.
Consider the case that $\ell(\bq)=2$.
Suppose by way of contradiction that $L_{\leq2}(\bu) \cap D_\bq(\bu)$ is empty.
Let $\varphi: P_f(X^+) \to S_{47}$ be a homomorphism such that $\varphi(x)=3$ for all $x\in c(\bq)$ and $\varphi(x)=2$ otherwise.
It is easy to see that $\varphi(\bu)= 2$, $\varphi(\bq)=1$ and so $\varphi(\bu)\neq \varphi(\bu+\bq)$, a contradiction.
Thus $L_{\leq2}(\bu) \cap D_\bq(\bu)$ is nonempty.
\end{proof}

\begin{pro}\label{pro38501}
$\mathsf{V}(S_{(4, 385)})$ is the ai-semiring variety defined by the identities
\begin{align}
xy& \approx yx; \label{38501}\\
x^2& \approx x^2+xy;    \label{38503}\\
xy& \approx xy+x_1x_2x_3; \label{38502}\\
x_1x_2x_3+y_1& \approx x_1x_2x_3+y_1+y_1y_2. \label{38504}
\end{align}
\end{pro}
\begin{proof}
It is easy to check that $S_{(4, 385)}$ satisfies the identities \eqref{38501}--\eqref{38504}.
In the remainder we need only prove that every ai-semiring identity of $S_{(4, 385)}$
can be derived by \eqref{38501}--\eqref{38504} and the identities defining $\mathbf{AI}$.
Let $\bu \approx \bu+\bq$ be such a nontrivial identity,
where $\bu=\bu_1+\bu_2+\cdots+\bu_n$ and $\bu_i, \bq\in X_c^+$, $1 \leq i \leq n$.
It is a routine matter to verify that $S_{(4, 385)}$ is isomorphic to
a subdirect product of $T_2$ and $S_{47}$.
So both $T_2$ and $S_{47}$ satisfy $\bu \approx \bu+\bq$.
This implies that $\ell(\bu_j)\geq 2$ for some $\bu_j\in \bu$.
By Lemma \ref{lem4701} we obtain that either $\ell(\bq)\geq 3$ or $\ell(\bq)=2$, $L_{\leq2}(\bu) \cap D_\bq(\bu)$ is nonempty.
If $\ell(\bq)\geq 3$, then
\[
\bu\approx \bu+\bu_j\stackrel{(\ref{38502})}\approx \bu+\bu_j+\bq\approx \bu+\bq.
\]
So we obtain $\bu \approx \bu+\bq$.
If $\ell(\bq)=2$ and $L_{\leq2}(\bu) \cap D_\bq(\bu)$ is nonempty,
then there exists $\bu_i\in D_\bq(\bu)$ such that $\ell(\bu_i)\leq 2$.
We only consider the following cases.

\textbf{Case 1.} $\ell(\bu_i)=1$. Then $\bu_i=h(\bq)$ or $\bu_i=t(\bq)$.
Without loss of generality, we may assume that $\bu_i=h(\bq)$. Then
\begin{align*}
\bu
&\approx \bu+\bu_i+\bu_j\\
&\approx \bu+h(\bq)+\bu_j+\bu_jh(\bq) &&(\text{by}~(\ref{38502}))\\
&\approx \bu+h(\bq)+\bu_j+\bu_jh(\bq)+h(\bq)t(\bq) &&(\text{by}~(\ref{38504}))\\
&\approx \bu+h(\bq)+\bu_j+\bu_jh(\bq)+\bq.
\end{align*}

\textbf{Case 2.} $\ell(\bu_i)=2$. Then $\bu_i=h(\bq)^2$ or $\bu_i=t(\bq)^2$.
Assume that $\bu_i=h(\bq)^2$. Then
\[
\bu \approx \bu+\bu_i \approx \bu+h(\bq)^2 \stackrel{(\ref{38503})}\approx \bu+h(\bq)^2+h(\bq)t(\bq) \approx \bu+h(\bq)^2+\bq.
\]
This derives $\bu \approx \bu+\bq$.
\end{proof}

\section{Equational bases of some subdirectly irreducible $4$-element ai-semirings}
In this section we concentrate on the finite basis problem for some 4-element ai-semirings that are subdirectly irreducible.

\begin{pro}\label{pro36001}
$\mathsf{V}(S_{(4, 360)})$ is the ai-semiring variety defined by the identities
\begin{align}
xyz & \approx xyz+y; \label{36001} \\
xyz & \approx xy+yz+xz; \label{36002}\\
xy+yz & \approx xy+yz+xz, \label{36003}
\end{align}
where $x$ and $z$ may be empty in $(\ref{36001})$.
\end{pro}
\begin{proof}
It is easily verified that $S_{(4, 360)}$ satisfies the identities \eqref{36001}--\eqref{36003}.
In the remainder it is enough to show that every ai-semiring identity of $S_{(4, 360)}$
can be derived by \eqref{36001}--\eqref{36003} and the identities defining $\mathbf{AI}$.
Let $\bu \approx \bu+\bq$ be such a nontrivial identity,
where $\bu=\bu_1+\bu_2+\cdots+\bu_n$ and $\bu_i, \bq\in X^+$, $1 \leq i \leq n$.
By the identity \eqref{36002}, we may assume that $\ell(\bq)\leq 2$ and $\ell(\bu_i)\leq 2$ for all $1 \leq i \leq n$.
Since $S_{54}$ is isomorphic to $\{1, 2, 4\}$ and $S_{57}$ is isomorphic to $\{1, 2, 3\}$,
we have that both $S_{54}$ and $S_{57}$ satisfy $\bu \approx \bu+\bq$.
It is easy to see that $S_{54}$ and $S_{57}$ have dual multiplications.
If $\ell(\bq)=1$, then by Lemma \ref{lem5701} there exists $\bu_i \in L_2(\bu)$
such that $\bu_i=\bq x$ or $\bu_i=x\bq$ for some $x\in X$.
By the identity \eqref{36001} we have
\[
\bu \approx \bu+\bu_i \approx \bu+\bu_i+\bq\approx \bu+\bq.
\]

Now consider the case that $\ell(\bq)=2$. Then $\bq=xy$ for some $x, y \in X$.
By Lemma \ref{lem5701} and its dual we have that
$L_2(\bu)$ contains $xx'$ and $y'y$ for some $x', y' \in X$.
In the following we shall think of $L_2(\bu)$ as a directed graph
whose vertex set is $c(L_2(\bu))$ and edge set consists of $(x_1, y_1)$ if $x_1y_1\in L_2(\bu)$.

\textbf{Case 1.} $x = y$.
If there is no path from $x$ to $x$ in the directed graph $L_2(\bu)$,
then we consider the semiring homomorphism $\varphi: P_f(X^+) \to S_{(4, 360)}$ defined by
$\varphi(x)=1$, $\varphi(z)=4$ if there is a path from $x$ to $z$, and $\varphi(z)=3$ otherwise.
Then $\varphi(\bu)= 1$ and $\varphi(\bq)=2$, a contradiction.
Thus there is a path from $x$ to $x$, and so $L_2(\bu)$ contains
$xx_1, x_1x_2, x_2x_3, \cdots, x_nx$ for some $x_1, x_2, \ldots, x_n\in X$ and $n\geq1$.
Now we have
\begin{align*}
\bu
&\approx \bu+xx_1+x_1x_2+x_2x_3+\cdots+x_{n}x\\
&\approx \bu+xx_1+x_1x_2+x_2x_3+\cdots+x_{n}x+x^2  &&(\text{by}~\eqref{36003})\\
&\approx \bu+xx_1+x_1x_2+x_2x_3+\cdots+x_{n}x+\bq.
\end{align*}
This derives $\bu \approx \bu+\bq$.

\textbf{Case 2.}  $x \neq y$. Suppose by way of contradiction that there is no path from
$x$ to $y$ in the directed graph $L_2(\bu)$.
Let $\varphi: P_f(X^+) \to S_{(4, 360)}$ be a semiring homomorphism such that $\varphi(z)=4$
if $z=x$ or there is a path from $x$ to $z$, and $\varphi(z)=3$ otherwise.
Then $\varphi(\bu)= 1$ and $\varphi(\bq)=2$, which is a contradiction.
So there is a path from $x$ to $y$.
By the identity \eqref{36003} we can derive $\bu \approx \bu+\bq$.
\end{proof}

It is easy to see that $S_{56}$ and $S_{58}$  have dual multiplications.
By \cite[Lemma 4.1]{yrzs} we immediately have
\begin{lem}\label{lem5601}
Let $\bu\approx \bu+\bq$ be a nontrivial ai-semiring identity such that $\bu=\bu_1+\bu_2+\cdots+\bu_n$
and $\bu_i, \bq \in X^+$, $1\leq i \leq n$.
Suppose that $\bu\approx \bu+\bq$ is satisfied by $S_{56}$. Then
$L_{\geq 2}(\bu) \neq \emptyset$, $T_{\bq}(\bu) \neq \emptyset$, and
$L_{\geq 2}(\bu)\cap T_{\bq}(\bu) \neq \emptyset$ if $\ell(\bq)\geq 2$.
\end{lem}

\begin{pro}\label{pro36301}
$\mathsf{V}(S_{(4, 363)})$ is the ai-semiring variety defined by the identities
\begin{align}
xy & \approx xy+y; \label{36301} \\
xyz & \approx xz+yz; \label{36302}\\
xy+yz& \approx xy+yz+xz. \label{36303}
\end{align}
\end{pro}
\begin{proof}
It is easy to check that $S_{(4, 363)}$ satisfies the identities \eqref{36301}--\eqref{36303}.
In the remainder we need only show that every ai-semiring identity of $S_{(4, 363)}$
is derivable from \eqref{36301}--\eqref{36303} and the identities defining $\mathbf{AI}$.
Let $\bu \approx \bu+\bq$ be such a nontrivial identity,
where $\bu=\bu_1+\bu_2+\cdots+\bu_n$ and $\bu_i, \bq\in X^+$, $1 \leq i \leq n$.
By the identity \eqref{36302} we may assume that $\ell(\bq)\leq 2$ and $\ell(\bu_i)\leq 2$ for all $\bu_i\in \bu$.
It is easy to see that $S_{56}$ and $S_{57}$ can be embedded into $S_{(4, 363)}$.
So both $S_{56}$ and $S_{57}$ satisfy $\bu \approx \bu+\bq$.
If $\ell(\bq)=1$, then by Lemma \ref{lem5601} $\bq=t(\bu_i)$ for some $\bu_i \in L_2(\bu)$, and so
\[
\bu \approx \bu+\bu_i\approx \bu+h(\bu_i)\bq \stackrel{(\ref{36301})}\approx  \bu+h(\bu_i)\bq+\bq\approx \bu+\bq.
\]

Now suppose that $\ell(\bq)=2$. Then $\bq=xy$ for some $x, y\in X$.
By Lemmas \ref{lem5701} and \ref{lem5601} we have that
$L_2(\bu)$ contains $xx'$ and $y'y$ for some $x', y'\in X$.
The remaining steps are similar to Case 1 and Case 2 in the proof of Proposition \ref{pro36001}.
\end{proof}

\begin{cor}
The ai-semiring $S_{(4, 376)}$ is finitely based.
\end{cor}
\begin{proof}
It is easy to see that $S_{(4, 376)}$ and $S_{(4, 363)}$ have dual multiplications.
By Proposition $\ref{pro36301}$ we obtain that $S_{(4, 376)}$ is finitely based.
\end{proof}

\begin{pro}\label{pro36601}
$\mathsf{V}(S_{(4, 366)})$ is the ai-semiring variety defined by the identities
\begin{align}
xy & \approx yx; \label{36601} \\
xy  & \approx xy+x; \label{36603} \\
xy& \approx xy+x^2; \label{36604}\\
x_1x_2x_3 & \approx x_1x_2x_3+y; \label{36602} \\
xy+yz& \approx xy+yz+xz. \label{36605}
\end{align}
\end{pro}
\begin{proof}
It is easy to check that $S_{(4, 366)}$ satisfies the identities \eqref{36601}--\eqref{36605}.
In the remainder we need only prove that every ai-semiring identity of $S_{(4, 366)}$
can be derived by \eqref{36601}--\eqref{36605} and the identities defining $\mathbf{AI}$.
Let $\bu \approx \bu+\bq$ be such a nontrivial identity,
where $\bu=\bu_1+\bu_2+\cdots+\bu_n$ and $\bu_i, \bq\in X_c^+$, $1 \leq i \leq n$.
Since $T_2$ is isomorphic to $\{1, 2\}$,
it follows that $T_2$ satisfies $\bu \approx \bu+\bq$ and so $\ell(\bu_j)\geq 2$ for some $\bu_j \in \bu$.
Since $S_{59}$ is isomorphic to $\{1, 2, 3\}$, we have that $S_{59}$ also satisfies $\bu \approx \bu+\bq$.
By Lemma \ref{lem5901} it is enough to consider the following two cases.

\textbf{Case 1.} $\ell(\bu_i)\geq 3$ for some $\bu_i \in \bu$. Then
\[
\bu \approx \bu+\bu_i \stackrel{(\ref{36602})}\approx \bu+\bu_i+\bq \approx \bu+\bq.
\]

\textbf{Case 2.} $\ell(\bu_i) \leq 2$ for all $\bu_i \in \bu$.
Then $\ell(\bq)\leq 2$ and so $c(\bq) \subseteq c(L_2(\bu))$.

\textbf{Subcase 2.1.} $\ell(\bq)=1$. Then there exists $\bu_j \in L_2(\bu)$ such that
$\bu_j=\bq x$ for some $x\in X$. Now we have
\[
\bu \approx \bu+\bq x \stackrel{(\ref{36603})}\approx  \bu+\bq x+\bq.
\]

\textbf{Subcase 2.2.} $\ell(\bq)=2$.
If $\bq$ is not linear, then $\bq=x^2$ for some $x\in X$
and so there exists $\bu_k \in L_2(\bu)$ such that $\bu_k = xy$ for some $y\in X$.
Now we have
\[
\bu \approx \bu+\bu_k \approx \bu+xy \stackrel{(\ref{36604})}\approx  \bu+xy+x^2  \approx \bu+xy+\bq.
\]
If $\bq$ is linear, then $\bq=xy$ for some distinct elements $x$ and $y$ of $X$.
By the identities \eqref{36601} and \eqref{36603}, we may assume that $c(L_1(\bu))\cap c(L_2(\bu))=\emptyset$.
Now $L_2(\bu)$ can be  thought of as a graph
whose vertex set is $c(L_2(\bu))$ and edge set consists of $\{x_1, y_1\}$ if $x_1y_1\in L_2(\bu)$.
Then both $x$ and $y$ are vertices of this graph. Suppose by way of contradiction that there is no path connecting
$x$ and $y$ in this graph. Consider the semiring homomorphism $\varphi: P_f(X^+) \to S_{(4, 366)}$ defined by $\varphi(z)=3$
if $z$ and $x$ are in the same connected component, and $\varphi(z)=4$ otherwise.
Then $\varphi(\bu)= 1$ and $\varphi(\bq)=2$, a contradiction.
It follows that there is a path connecting $x$ and $y$. So we may assume that $xx_1, x_1x_2, x_2x_3, \cdots, x_{k+1}y \in L_2(\bu)$
for some $x_1, x_2, \ldots, x_{k+1}\in X$ and $k\geq0$. Now we have
\begin{align*}
\bu
&\approx \bu+xx_1+x_1x_2+x_2x_3+\cdots+x_{k+1}y\\
&\approx \bu+xx_1+x_1x_2+x_2x_3+\cdots+x_{k+1}y+xy  &&(\text{by}~\eqref{36605})\\
&\approx \bu+xx_1+x_1x_2+x_2x_3+\cdots+x_{k+1}y+\bq.
\end{align*}
This implies the identity $\bu \approx \bu+\bq$ as required.
\end{proof}

The following result can be found in \cite[Lemma 6.1]{yrzs}.
\begin{lem}\label{lem6001}
Let $\bu\approx \bu+\bq$ be a nontrivial ai-semiring identity such that
$\bu=\bu_1+\bu_2+\cdots+\bu_n$
and $\bu_i, \bq \in X^+$, $1\leq i \leq n$.
Suppose that $\bu\approx \bu+\bq$ is satisfied by $S_{60}$.
Then $L_{\geq 2}(\bu)\neq \emptyset$ and $c(\bq)\subseteq c(L_{\geq2}(\bu))$.
\end{lem}
\begin{pro}\label{pro36801}
$\mathsf{V}(S_{(4, 368)})$ is the ai-semiring variety defined by the identities
\begin{align}
xy & \approx yx; \label{36801} \\
xy & \approx xy+x; \label{36802} \\
xy  & \approx xy+x^2; \label{36803} \\
xy+yz& \approx xy+yz+xz; \label{36804}\\
xy+yy_1y_2& \approx xyy_1y_2; \label{36805}\\
x_1x_2x_3+y_1y_2y_3& \approx x_1x_2x_3y_1y_2y_3. \label{36806}
\end{align}
\end{pro}
\begin{proof}
It is easy to verify that $S_{(4, 368)}$ satisfies the identities \eqref{36801}--\eqref{36806}.
In the remainder we shall show that every ai-semiring identity of $S_{(4, 368)}$
is derivable from \eqref{36801}--\eqref{36806} and the identities defining $\mathbf{AI}$.
Let $\bu \approx \bu+\bq$ be such a nontrivial identity,
where $\bu=\bu_1+\bu_2+\cdots+\bu_n$ and $\bu_i, \bq\in X_c^+$, $1 \leq i \leq n$.
Since $S_{59}$ is isomorphic to $\{1, 2, 3\}$, it follows that $S_{59}$ satisfies the identity $\bu \approx \bu+\bq$.
Since $S_{60}$ is isomorphic to $\{1, 2, 4\}$, we have that $S_{60}$ also satisfies the identity $\bu \approx \bu+\bq$.
By Lemmas \ref{lem5901} and \ref{lem6001} it is enough to consider the following three cases.

\textbf{Case 1.} $\ell(\bq)=1$. Then $c(\bq)\subseteq c(\bu)$ and so
there exists $\bu_k\in L_{\geq 2}(\bu)$ such that $\bu_k=\bq\bu_k'$ for some $\bu_k' \in X^+$.
Now we have
\[
\bu \approx \bu+\bu_k \approx \bu+\bq\bu_k' \stackrel{\eqref{36802}}\approx  \bu+\bq\bu_k'+\bq.
\]
This implies the identity $\bu \approx \bu+\bq$.

\textbf{Case 2.} $\ell(\bq)=2$.
Then $c(\bq) \subseteq  c(L_{\geq 2}(\bu))$.
Since $S_{(4, 368)}$ satisfies the identities \eqref{36802}, \eqref{36805} and \eqref{36806},
we may assume that
\[
c(L_1(\bu)) \cap c(L_{\geq 2}(\bu))=c(L_2(\bu)) \cap c(L_{\geq 3}(\bu))=\emptyset.
\]
If $|c(\bq)|=1$, then $\bq=x^2$ for some $x\in X$.
So there exists $\bu_j \in c(L_{\geq 2}(\bu))$ such that
$\bu_j = x\bu_j'$ for some $\bu_j' \in X^+$. Now we have
\[
\bu \approx \bu+\bu_j\approx \bu+x\bu_j' \stackrel{\eqref{36803}}\approx  \bu+x\bu_j'+x^2  \approx \bu+x\bu_j'+\bq \approx \bu+\bq.
\]
If $|c(\bq)|=2$, then $\bq=xy$ for some distinct elements $x$ and $y$ of $X$.
By Lemma \ref{lem6001} we have that $c(\bq)\subseteq c(L_{\geq2}(\bu))$.

\textbf{Subcase 2.1.} $c(\bq) \subseteq c(L_{\geq 3}(\bu))$.
Then there $\bu_i, \bu_j\in L_{\geq 3}(\bu)$ such that $x\in c(\bu_i)$ and $y\in c(\bu_j)$.
This implies that $\bu_i\bu_j=\bq\bq_1$ for some $\bq_1 \in X^+$ and so
\[
\bu\approx \bu+\bu_i+\bu_j \stackrel{\eqref{36806}}\approx\bu+\bu_i\bu_j\approx\bu+\bq\bq_1 \stackrel{\eqref{36802}}\approx\bu+\bq\bq_1+\bq.
\]

\textbf{Subcase 2.2.} $c(\bq)\subseteq c(L_2(\bu))$.
Let us think of $L_2(\bu)$ as a graph.
Suppose by way of contradiction that there is no path connecting
$x$ and $y$ in this graph. Consider the semiring homomorphism $\varphi: P_f(X^+) \to S_{(4, 368)}$ such that $\varphi(z)=3$
if $z$ and $x$ are in the same connected component, and $\varphi(z)=4$ otherwise.
It is easy to verify that $\varphi(\bu)= 1$ and $\varphi(\bq)=2$, a contradiction.
Thus there is a path connecting $x$ and $y$.
So we may assume that $xx_1, x_1x_2, x_2x_3, \cdots, x_{k+1}y \in L_2(\bu)$
for some $x_1, x_2, \ldots, x_{k+1}\in X$.
By \eqref{36804} we can derive the identity $\bu \approx \bu+\bq$.

\textbf{Subcase 2.3.}
$x\in c(L_2(\bu))$, $y\in c(L_{\geq 3}(\bu))$.
Let $\varphi: P_f(X^+) \to S_{(4, 368)}$ a semiring homomorphism
defined by $\varphi(z)=3$ if $z \in c(L_2(\bu))$, and $\varphi(z)=4$ otherwise.
Then $\varphi(\bu)=1$ and $\varphi(\bq)=2$, a contradiction.

\textbf{Case 3.} $\ell(\bq)\geq 3$. We shall show that $c(\bq) \subseteq  c(L_{\geq 3}(\bu))$.
Indeed, suppose that this is not true. Then
there exists $x \in c(\bq)$, $x \in c(L_2(\bu))$, but $x \notin c(L_{\geq 3}(\bu))$.
Consider the semiring homomorphism $\psi: P_f(X^+) \to S_{(4,368)}$
defined by $\psi(z)=3$ if $z \in c(L_2(\bu))$, and $\psi(z)=4$ otherwise.
It is easy to see that $\psi(\bu)=1$ and $\psi(\bq)=2$.
This is a contradiction. So $c(\bq) \subseteq  c(L_{\geq 3}(\bu))$.
The remaining steps are similar to Subcase 2.1.
\end{proof}
\begin{pro}\label{pro36901}
$\mathsf{V}(S_{(4, 369)})$ is the ai-semiring variety defined by the identities
\begin{align}
xy & \approx yx; \label{36901} \\
x^2 & \approx x^2+y; \label{36904}\\
xy& \approx xy+x; \label{36905}\\
x_1x_2x_3 & \approx x_1x_2x_3+x_4; \label{36902} \\
x_1x_2+x_2x_3+x_3x_4&\approx x_1x_2+x_2x_3+x_3x_4+x_1x_4. \label{36906}
\end{align}
\end{pro}
\begin{proof}
It is easy to check that $S_{(4, 369)}$ satisfies the identities \eqref{36901}--\eqref{36906}.
In the remainder we need only prove that every ai-semiring identity of $S_{(4, 369)}$
is derivable from \eqref{36901}--\eqref{36906} and the identities defining $\mathbf{AI}$.
Let $\bu \approx \bu+\bq$ be such a nontrivial identity,
where $\bu=\bu_1+\bu_2+\cdots+\bu_n$ and $\bu_i, \bq\in X_c^+$, $1 \leq i \leq n$.
Since $T_2$ is isomorphic to $\{1, 2\}$,
it follows that $T_2$ satisfies $\bu \approx \bu+\bq$ and so $L_{\geq 2}(\bu)$ is nonempty.
If $\ell(\bu_i)\geq 3$ for some $\bu_i \in \bu$, then
\[
\bu
\approx \bu+\bu_i\stackrel{\eqref{36902}}\approx \bu+\bu_i+\bq\approx \bu+\bq.
\]

Now suppose that $\ell(\bu_i) \leq 2$ for all $\bu_i \in \bu$.
By \eqref{36901} and \eqref{36905} we may assume that $c(L_1(\bu))\cap c(L_2(\bu))=\emptyset$.
Now let us think of $L_2(\bu)$ as a graph.
If there exists a loop in this graph, then there exists $\bu_i \in \bu$ such that $\bu_i=x^2$ for some $x\in X$.
So we have
\[
\bu
\approx \bu+\bu_i
\approx \bu+x^2
\stackrel{\eqref{36904}}\approx \bu+x^2+\bq\approx \bu+\bq.
\]
If there exists an odd circle in this graph,
we may assume that $L_2(\bu)$ contains
$xx_1, x_1x_2,$ $ x_2x_3, \ldots, x_{2k}x$
for some $x_1, x_2, \ldots, x_{2k}\in X$ and $k\geq1$.
Then
\begin{align*}
\bu
&\approx \bu+xx_1+x_1x_2+x_2x_3+\cdots+x_{2k}x\\
&\approx \bu+xx_1+x_1x_2+x_2x_3+\cdots+x_{2k}x+x^2  &&(\text{by}~\eqref{36906})\\
&\approx \bu+xx_1+x_1x_2+x_2x_3+\cdots+x_{2k}x+x^2+\bq.  &&(\text{by}~\eqref{36904})
\end{align*}
This derives the identity $\bu \approx \bu+\bq$.

In the following suppose that there are no loops or odd circles in the graph $L_2(\bu)$.
Recall that a graph is \emph{bipartite} if its vertex set can be decomposed into two disjoint sets such that no two
vertices within the same set are adjacent.
From \cite[Theorem 4]{bol98} we know that $L_2(\bu)$ is a bipartite graph
and so there exist $A$ and $B$ such that $A\cup B=c(L_2(\bu))$, $A\cap B=\emptyset$
and no edges exist between vertices within $A$ or $B$.

Next, we shall show that $\ell(\bq) \leq 2$ and $c(\bq) \subseteq c(\bu)$.
In fact, let $\varphi : P_f(X^+) \rightarrow S_{(4, 369)}$ be a semiring homomorphism
defined by $\varphi(x)=3$ if $x\in A$, $\varphi(x)=4$ if $x\in c(L_1(\bu))\cup B$,
and $\varphi(x)=2$ otherwise.
It is easy to see that $\varphi(\bu)=1$ and so $\varphi(\bq)\leq 1$.
This implies that $c(\bq) \subseteq c(\bu)$.
Since \eqref{36902} holds in $S_{(4, 369)}$, it follows that $\ell(\bq) \leq 2$.

If $\ell(\bq)=1$, then $\bq \in c(L_2(\bu))$ and so there exists $\bu_k \in L_2(\bu)$ such that
$\bu_k=\bq x$ for some $x\in X$. By \eqref{36905} we deduce
\[
\bu\approx \bu+\bu_k\approx \bu+\bq x\approx \bu+\bq x+\bq\approx \bu+\bq.
\]
If $\ell(\bq)=2$, one can use the above substitution $\varphi$ to show that $\bq$ is a linear word.
Let us write $\bq=xy$.
Let $\psi: P_f(X^+) \to S_{(4, 369)}$ be a semiring homomorphism defined by
$\psi(z)=3$ if $z \in A$, $\psi(z)=4$ if $z \in B$, and $\psi(z)=1$ otherwise.
It is easy to see that $\psi(\bu)=1$ and so $\psi(\bq)=1$.
This implies that $x\in A$, $y\in B$ or $x\in B$, $y\in A$.
Without loss of generality, assume that $x\in A$ and $y\in B$.
We shall show that there is a path connecting $x$ and $y$.
Suppose that it is not true.
Consider the substitution $\beta: P_f(X^+) \to S_{(4, 369)}$
defined by $\beta(z)=3$ if $z\in A\setminus y^*$ or $ z \in B \cap y^*$,
and $\beta(z)=4$ otherwise, where $y^*$ denotes the connected component of $y$.
Then $\beta(\bu)= 1$ and $\beta(\bq)=2$, a contradiction.
It follows that there is a path connecting $x$ and $y$, and so $L_2(\bu)$ contains
$xx_1, x_1x_2,$ $ x_2x_3, \ldots, x_{2k}y$
for some $x_1, x_2, \ldots, x_{2k}\in X$ and $k\geq1$.
By \eqref{36906} we can derive the identity $\bu \approx \bu+\bq$.
\end{proof}

\begin{pro}\label{pro36901}
$\mathsf{V}(S_{(4, 370)})$ is the ai-semiring variety defined by the identities
\begin{align}
x^2  & \approx x^2+y; \label{37003} \\
xy & \approx xy+x; \label{37004}\\
xy& \approx xy+y; \label{37005}\\
xy+yz &\approx xyz; \label{37006}\\
x_1x_2x_3 & \approx x_1x_2x_3+x_4; \label{37001} \\
x_1x_2+x_3x_4  & \approx x_1x_2+x_3x_4+x_1x_4. \label{37007}
\end{align}
\end{pro}
\begin{proof}
It is easy to verify that $S_{(4, 370)}$ satisfies the identities \eqref{37003}--\eqref{37007}.
In the remainder it is enough to prove that every ai-semiring identity of $S_{(4, 370)}$
can be derived by \eqref{37003}--\eqref{37007} and the identities defining $\mathbf{AI}$.
Let $\bu \approx \bu+\bq$ be such a nontrivial identity,
where $\bu=\bu_1+\bu_2+\cdots+\bu_n$ and $\bu_i, \bq\in X^+$, $1 \leq i \leq n$.
Since $T_2$ is isomorphic to $\{1, 2\}$, it follows that $T_2$ satisfies $\bu \approx \bu+\bq$ and so $L_{\geq 2}(\bu)$ is nonempty.
If $\ell(\bu_i)\geq 3$ for some $\bu_i \in \bu$, then
\[
\bu
\approx \bu+\bu_i \stackrel{(\ref{37001})} \approx \bu+\bu_i+\bq \approx \bu+\bq.
\]
Now suppose that $\ell(\bu_i) \leq 2$ for all $\bu_i \in \bu$.
By \eqref{37004} and \eqref{37005} we may assume that $c(L_1(\bu)) \cap c(L_2(\bu))=\emptyset$.
We consider the following two cases.

\textbf{Case 1.} $t(\bu_k)=h(\bu_\ell)$ for some $\bu_k,\bu_\ell \in L_2(\bu)$. Then
\begin{align*}
\bu
&\approx \bu+\bu_k+\bu_\ell\\
&\approx \bu+h(\bu_k)t(\bu_k)+h(\bu_\ell)t(\bu_\ell)\\
&\approx \bu+h(\bu_k)t(\bu_k)+t(\bu_k)t(\bu_\ell)\\
&\approx \bu+h(\bu_k)t(\bu_k)t(\bu_\ell) &&(\text{by}~\eqref{37006})\\
&\approx \bu+h(\bu_k)t(\bu_k)t(\bu_\ell)+\bq.&&(\text{by}~\eqref{37001})
\end{align*}
This implies the identity $\bu\approx \bu+\bq$.

\textbf{Case 2.} $t(\bu_k)\neq h(\bu_\ell)$ for all $\bu_k,\bu_\ell \in L_2(\bu)$.
We shall show that $\ell(\bq) \leq 2$, $\bq$ is a linear word, $c(\bq) \subseteq c(\bu)$,
and $h(\bq) \in h(L_2(\bu))$ and $t(\bq) \in t(L_2(\bu))$ if $\ell(\bq) = 2$.
Indeed, let $\varphi : P_f(X^+) \rightarrow S_{(4,370)}$ be a semiring homomorphism
defined by $\varphi(x)=4$ if $x\in h(L_2(\bu))$, $\varphi(x)=3$ if $x\in t(L_2(\bu))$,
$\varphi(x)=1$ if $x\in c(L_1(\bu))$, and $\varphi(x)=2$ otherwise.
It is easy to see that $\varphi(\bu)=1$ and so $\varphi(\bq)\leq 1$.
This implies that $c(\bq) \subseteq c(\bu)$.
Since \eqref{37003} and \eqref{37001} hold in $S_{(4, 370)}$, it follows that $\bq$ is a linear word and $\ell(\bq) \leq 2$.
If $\ell(\bq) = 2$, then $\varphi(\bq)=1$ and so $h(\bq) \in h(L_2(\bu))$ and $t(\bq) \in t(L_2(\bu))$.

If $\ell(\bq)=1$, then there exists $\bu_j \in L_2(\bu)$
such that $\bu_j=\bq t(\bu_j)$ or $\bu_j=h(\bu_j)\bq$. So we have
\[
\bu  \approx \bu+\bu_j \approx \bu+\bq t(\bu_j) \stackrel{(\ref{37004})}\approx \bu+\bq t(\bu_j)+\bq \approx \bu+\bq
\]
or
\[
\bu  \approx \bu+\bu_j \approx \bu+h(\bu_j)\bq \stackrel{(\ref{37005})}\approx \bu+h(\bu_j)\bq+\bq\approx \bu+\bq.
\]

If $\ell(\bq)=2$, we may write $\bq=xy$ for some distinct elements $x$ and $y$ of $X$.
Then $x\in h(L_2(\bu))$ and $y\in t(L_2(\bu))$.
This implies that $xx_1$ and $y_1y$ are both in $L_2(\bu)$ for some elements $x_1$ and $y_1$ of $X$.
By the identity \eqref{37007} we have
\[
\bu \approx  \bu+xx_1+y_1y \approx \bu+xx_1+y_1y+xy\approx \bu+xx_1+y_1y+\bq
\]
and so the identity $\bu \approx \bu+\bq$ is derived.
\end{proof}

\begin{pro}\label{pro37501}
$\mathsf{V}(S_{(4, 375)})$ is the ai-semiring variety defined by the identities
\begin{align}
xy & \approx yx; \label{37501} \\
xy & \approx xy+x; \label{37502}\\
xy &\approx xy+x^2; \label{37504}\\
xyz& \approx xy+yz+xz; \label{37503}\\
xy+yz  & \approx xy+yz+xz. \label{37505}
\end{align}
\end{pro}
\begin{proof}
It is easy to verify that $S_{(4, 375)}$ satisfies the identities \eqref{37501}--\eqref{37505}.
In the remainder we shall show that every ai-semiring identity of $S_{(4, 375)}$
is derivable from \eqref{37501}--\eqref{37505} and the identities defining $\mathbf{AI}$.
Let $\bu \approx \bu+\bq$ be such a nontrivial identity,
where $\bu=\bu_1+\bu_2+\cdots+\bu_n$ and $\bu_i, \bq\in X_c^+$, $1 \leq i \leq n$.
By the identities \eqref{37503} and \eqref{37502},
we may assume that $\ell(\bq)\leq 2$, $\ell(\bu_i)\leq 2$ for all $1 \leq i \leq n$,
and $c(L_1(\bu))\cap c(L_2(\bu))=\emptyset$.
Since $S_{60}$ is isomorphic to $\{1, 2, 3\}$, it follows that $S_{60}$ satisfies $\bu \approx \bu+\bq$.
By Lemma \ref{lem6001} we have that $L_2(\bu)\neq \emptyset$ and $c(\bq)\subseteq c(L_2(\bu))$.

If $\ell(\bq)=1$, then $\bu_i=\bq x$ for some $x\in X$, and so
\[
\bu \approx \bu+\bu_i \approx \bu+\bq x \stackrel{\eqref{37502}}\approx  \bu+\bq x+\bq\approx  \bu+\bq.
\]
Now suppose that $\ell(\bq)=2$. Then $\bq=xy$ for some $x, y\in X$.

\textbf{Case 1.} $x = y$. Then $\bu_j=x t(\bu_j)$ for some $\bu_j\in L_2(\bu)$.
Furthermore, we have
\[
\bu \approx \bu+\bu_j \approx \bu+x t(\bu_j) \stackrel{\eqref{37504}}\approx \bu+x t(\bu_j)+x^2\approx \bu+\bq.
\]

\textbf{Case 2.} $x \neq y$.
Now we think of $L_2(\bu)$ as a graph.
Suppose by way of contradiction that $x$ and $y$ are not in the same connected component.
Consider the semiring homomorphism $\varphi: P_f(X^+) \to S_{(4, 375)}$ such that $\varphi(z)=3$
if $z$ and $x$ are in the same connected component, $\varphi(z)=4$ otherwise.
Then $\varphi(\bu)= 1$ and $\varphi(\bu+\bq)=2$.
This is a contradiction and so there is a path connecting $x$ and $y$.
So we may assume that $L_2(\bu)$ contains $xx_1, x_1x_2, x_2x_3, \ldots, x_ny$
for some $x_1, x_2, \ldots, x_n\in X$ and $n\geq1$.
By the identity \eqref{37505} we can derive the identity $\bu \approx \bu+\bq$.
\end{proof}

\qquad

\noindent
\textbf{Acknowledgements}
The authors would like to extend their gratitude to their team members
Zidong Gao, Yanling Liang, Junyang Liu, Qizheng Sun, Xin Ye, Mengyu Yuan, and Lingli Zeng for their productive discussions.
We are profoundly grateful to Professor Mikhail V. Volkov for his unwavering long-term support, encouragement, and invaluable assistance.
We also wish to thank the anonymous referee for her/his valuable comments and suggestions,
which have significantly enhanced the quality of this paper.

\bibliographystyle{amsplain}

\end{document}